\documentclass[a4paper,11pt,abstracton,normalheadings,headsepline,headinclude,footexclude,bibtotoc,oneside]{scrartcl}
\RequirePackage{amsmath}                        
\RequirePackage{amsthm}
\RequirePackage{amssymb}                        
\RequirePackage{latexsym}                       
\RequirePackage{stmaryrd}                       
\RequirePackage[all]{xy}                        
\UseTips                                    
\newdir^{c}{{}*!/-3pt/\dir^{(}}
\newdir_{c}{{}*!/-3pt/\dir_{(}}

\RequirePackage{xspace}                         
\RequirePackage{calc}                           
\RequirePackage{ifthen}                         

\RequirePackage{mathrsfs}
\renewcommand{\mathcal}[1]{\mathscr{#1}}

\RequirePackage[dvipsnames,usenames]{color}  

\RequirePackage{hyperref}
\hypersetup{
bookmarks,
bookmarksdepth=3,
bookmarksopen,
bookmarksnumbered,
pdfstartview=FitH,
colorlinks,backref,hyperindex,
linkcolor=RawSienna,
anchorcolor=BurntOrange,
citecolor=OliveGreen,
filecolor=BlueViolet,
menucolor=Yellow,
urlcolor=OliveGreen
}

\usepackage[german,english]{babel}      
\usepackage[ansinew]{inputenc}

\makeatletter
\@ifdefinable\equationname{\let\equationname\equationautorefname}
\def\equationautorefname~#1\@empty\@empty\null{(#1\@empty\@empty\null)}%
\@ifdefinable\itemname{\let\itemname\itemautorefname}
\def\itemautorefname~#1\@empty\@empty\null{(#1\@empty\@empty\null)}%
\makeatother

\usepackage{aliascnt}
\newcommand{\basetheorem}[3]{\newtheorem{#1}{#2}[#3] \newtheorem*{#1*}{#2} \newcommand{\thmautorefname}{Theorem}}%
\newcommand{\maketheorem}[3]{%
    \newaliascnt{#1}{#3}
    \newtheorem{#1}[#1]{#2}
    \aliascntresetthe{#1}
    \expandafter\def\csname #1autorefname\endcsname{#2}
    \newtheorem*{#1*}{#2}
}%
\theoremstyle{plain}   

\basetheorem{thm}{Theorem}{section}

\maketheorem{mainthm}{Main~Theorem}{thm}
\maketheorem{prop}{Proposition}{thm}
\maketheorem{cor}{Corollary}{thm}
\maketheorem{lem}{Lemma}{thm}
\maketheorem{conj}{Conjecture}{thm}

\theoremstyle{definition}    

\maketheorem{defn}{Definition}{thm}

\theoremstyle{remark}    

\maketheorem{rem}{Remark}{thm}
\maketheorem{ex}{Example}{thm}
\maketheorem{ques}{Question}{thm}

\numberwithin{equation}{section}
\numberwithin{thm}{section}

\makeatletter
\renewcommand{\theenumi}{\alph{enumi}}

\renewcommand{\p@enumii}{\theenumi-}

\renewcommand{\labelitemi}{$\m@th\circ$}
\renewcommand{\labelitemii}{$\m@th\diamond$}
\renewcommand{\labelitemiii}{$\m@th\star$}
\renewcommand{\labelitemiv}{$\m@th\cdot$}
\makeatother

\newcommand{\ie}{\textit{i.e.}}

\newcommand{\cf}{\textit{cf.}}
\newcommand{\blank}{\phantom{m}}

\newcommand{\Fr}{{F}}
\newcommand{\Primefield}{\mathbb{F}}
\newcommand{\CO}{\mathcal{O}}
\newcommand{\CM}{\mathcal{M}}

\newcommand{\Hom}{\operatorname{Hom}}

\newcommand{\shom}{\mathcal{H}\!\mathit{om}}
\newcommand{\rshom}{\mathit{R}\shom}
\newcommand{\SHom}{\operatorname{\shom}}
\newcommand{\RSHom}{\operatorname{\rshom}}
\newcommand{\End}{\operatorname{End}}
\newcommand{\usc}{\underline{\phantom{M}}}
\newcommand{\tensor}{\otimes}
\newcommand{\Spec}{\operatorname{Spec}}
\newcommand{\defeq}{\stackrel{\scriptscriptstyle \operatorname{def}}{=}}
\newcommand{\Gen}{\mathsf{Gen}}

\renewcommand{\min}{\text{min}}
\renewcommand{\to}[1][]{\xrightarrow{\ #1\ }}
\newcommand{\onto}[1][]{\protect{\xrightarrow{\ #1\ }\hspace{-0.8em}\rightarrow}}
\newcommand{\into}[1][]{\lhook \joinrel \xrightarrow{\ #1\ }}
\renewcommand{\frm}{\frak{m}}
\newcommand{\frp}{\frak{p}}

\renewcommand{\phi}{\varphi}
\renewcommand{\min}{\mathrm{min}}

\newcommand{\CN}{\mathcal{N}}

\newcommand{\id}{\operatorname{id}}
\newcommand{\supp}{\operatorname{Supp}}
\newcommand{\Ann}{\operatorname{Ann}}
\newcommand{\BF}{\mathbb{F}}
\newcommand{\Perv}{\operatorname{Perv}}
\renewcommand{\th}{\ensuremath{^{\text{th}}}}
\newcommand{\nilord}{\operatorname{nilord}}
\newcommand{\suppcrys}{\operatorname{Supp_{crys}}}
\newcommand{\ql}{\operatorname{ql}}
\newcommand{\length}{\operatorname{length}}

\newcommand{\Cart}{\operatorname{\sf Cart}}
\newcommand{\CohCart}{\operatorname{\sf CohCart}}
\newcommand{\Crys}{\operatorname{\sf Crys}}
\newcommand{\MinCart}{\operatorname{\sf MinCart}}
\newcommand{\Tr}{\operatorname{Tr}}
\newcommand{\Ca}{{C}}
\newcommand{\Cadj}{\kappa}

\newcommand{\Cech}{\v{C}ech}

\reversemarginpar
\makeatletter
\advance\marginparwidth by \ta@bcor
\makeatother
\newcounter{themargin}
\def\m?#1{\textcolor{Mahogany}{\textbf{???$^{\text{\arabic{themargin}}}$}}{\marginpar{\footnotesize\color{Mahogany}\fbox{\parbox{\marginparwidth}{\textbf{mnl --- \arabic{themargin} ---}\addtocounter{themargin}{1}\\ #1}}} \immediate\write16{}%
\immediate\write16{Warning: There was still a question mark . . . }%
\immediate\write16{}}}
\def\mf?#1{\textcolor{Mahogany}{\footnotesize\newline{\color{Mahogany}\fbox{\parbox{\textwidth-5pt}{\textbf{mnl: } #1}}}}}

\definecolor{WarmDarkGray}{rgb}{0.25,0.2,0.18}

\newcommand{\comment}[1]{}

\newcommand{\nil}{\operatorname{nil}}
\newcommand{\CS}{\mathcal{S}}

\begin{document}

\title{Cartier modules: finiteness results}
\author{Manuel Blickle and Gebhard Böckle}
\maketitle
\footnotetext{\textbf{Address:} Fakultät für Mathematik, Universit\"at Duisburg-Essen, 45117 Essen, Germany}
\footnotetext{\textbf{email:} \texttt{manuel.blickle@gmail.com}, \texttt{boeckle@uni-due.de}}
\footnotetext{\textbf{Funding:} The first author is fully supported by a Heisenberg grant of the DFG. Both authors are partially supported by the DFG SFB/TR45 ``Periods, moduli spaces and the arithmetic of algebraic varieties''}
\footnotetext{\textit{Keywords and phrases:} Frobenius, $p^{-1}$-linear maps, Cartier operator}
\footnotetext{\textit{2010 Mathematics subject classification.} Primary: 14G17; Secondary: 13A35}

\begin{abstract}
 On a locally Noetherian scheme $X$ over a field of positive characteristic $p$ we study the category of coherent $\CO_X$-modules $M$ equipped with a $p^{-e}$-linear map, \ie~an additive map $\Ca: \CO_X \to \CO_X$ satisfying $r\Ca(m)=\Ca(r^{p^e}m)$ for all $m \in M, r \in \CO_X$. The notion of nilpotence, meaning that some power of the map $\Ca$ is zero, is used to rigidify this category. The resulting quotient category, called Cartier crystals, satisfies some strong finiteness conditions. The main reasult in this paper states that, if the Frobenius morphism on $X$ is a finite map, \ie~if $X$ is $F$-finite, then all Cartier crystals have finite length. We further show how this and related results can be used to recover and generalize other finiteness results of Hartshorne-Speiser \cite{HaSp}, Lyubeznik \cite{Lyub}, Sharp \cite{Sharp.GradedAnnihil}, Enescu-Hochster \cite{EnescuHochster.FrobLocCohom}, and Hochster \cite{Hochster.FinitenessFmod} about the structure of modules with a left action of the Frobenius. For example, we show that over any regular $F$-finite scheme $X$ Lyubeznkik's $F$-finite modules \cite{Lyub} have finite length.
\end{abstract}


\section{Introduction}
The Frobenius morphism is the fundamental tool in algebra and geometry over a field of positive characteristic. For a scheme $X$ over the finite field $\Primefield_q$ of $q=p^e$ elements, with $p$ a prime number, the Frobenius morphism is the unique $\Primefield_q$-linear map that sends sections of the structure sheaf $\CO_X$ to their $q\th$ power. In other words, there is a natural \emph{left} action of the Frobenius $\Fr$ on the structure sheaf $\CO_X$ satisfying $\Fr(rs)=r^q\Fr(s)$. An additive map on an arbitrary $\CO_X$-module satisfying such a relation is called a $q$-linear map. From the example of the structure sheaf $\CO_X$ with its natural Frobenius action, many quasi-coherent sheaves $M$ on $X$ receive a natural \emph{left} action of $\Fr$, or equivalently are equipped with a map $M \to \Fr_*M$. Important examples are the local cohomology modules: the study of their natural Frobenius action has lead to deep structural results \cite{HuSha.LocCohom,HaSp,Lyub} which are important for the study of singularities, amongst other things \cite{Smith.rat,Smith.sing,Hara,Bli.int,BliBon.LocCohomMult}.

In this article we consider a dual situation, namely that of a quasi-coherent $\CO_X$-module $M$ equipped with a \emph{right} action of the Frobenius $\Fr$. That is, $M$ is equipped with an $\CO_X$-linear map $\Ca=\Ca_M: \Fr_*M \to M$. Such a map is called \emph{$q^{-1}$-linear}, or \emph{Cartier linear}. We call a pair $(M,C)$ consisting of a quasi-coherent $\CO_X$-module $M$ and a $q^{-1}$-linear map $C$ on $M$ a \emph{Cartier module}. The prototypical example of a Cartier module is the dualizing sheaf $\omega_X$ of a smooth and finite type $k$-scheme $X$ ($k$ perfect) together with the classical Cartier operator, \cite{Cartier57,Katz}. The relation to the \emph{left} actions described above is via Grothendieck-Serre duality where the left action on $\CO_X$ corresponds to the Cartier map on $\omega_X$, \cf~\cite{BliBoe.CartierCrys}.

\subsection*{Finite length of Cartier modules up to nilpotence}
In this article we develop the basic theory of Cartier modules and derive a crucial structural result; namely, if $X$ is an $F$-finite scheme, then
\begin{quote}
    every coherent Cartier module has -- up to nilpotence -- finite length.
\end{quote}
A Cartier module $M=(M,\Ca)$ is called \emph{nilpotent}, if $\Ca^e=0$ for some $e \geq 0$.  This notion of nilpotence is crucial in our treatment. In fact, we study a theory of coherent Cartier modules \emph{up to nilpotence}. To formalize this viewpoint one localizes the abelian category of \emph{coherent} Cartier modules at its Serre subcategory of nilpotent Cartier modules. The localized category, which we call \emph{Cartier crystals}, is an abelian category. Its objects are coherent Cartier modules, but the notion of isomorphism has changed: A morphism $\phi: M \to N$ of Cartier modules induces an isomorphism of associated Cartier crystals if $\phi$ is a \emph{nil-isomorphism}, \ie~both, kernel and cokernel of $\phi$ are nilpotent. Thus, more precisely, the main result of this paper states:

\begin{mainthm*}[\autoref{t.FiniteLenth}]
If $X$ is a locally Noetherian scheme over $\Primefield_q$ such that the Frobenius $\Fr$ on $X$ is a finite map, then every coherent Cartier module has -- up to nilpotence -- finite length. More precisely, every coherent Cartier crystal has finite length.
\end{mainthm*}

\begin{proof}[Idea of the proof] Let us isolate the main points in the proof. The ascending chain condition is clear since $X$ is Noetherian and the underlying $\CO_X$-module of the Cartier module is coherent. Hence one has to show that any descending chain of Cartier submodules $M \supseteq M_1 \supseteq M_2 \supseteq M_3 \supseteq \ldots$ stabilizes up to nilpotence. This is shown by induction on the dimension of $X$. Two main ideas enter into its proof.
\begin{enumerate}
\item \label{z.ingrA}(\autoref{t.ImagesStabilize}) Any coherent Cartier module $(M,C)$ is nil-isomorphic to a Cartier submodule $(\underline{M},C)$ with surjective structural map, \ie~with $C(\underline{M})=\underline{M}$. It is easy to see that the support of $\underline{M}$ is a \emph{reduced}\footnote{Throughout this article, when we speak  of the support of a sheaf $M$ of $\CO_X$-modules we mean the subscheme of $X$ whose ideal sheaf is the annihilator of $M$.} subscheme of $X$.
\end{enumerate}
This part is proved by showing that the descending chain of images $M \supseteq C(M) \supseteq C^2(M) \supseteq \ldots$ stabilizes and this stable image is then $\underline{M}$ (see \autoref{t.ImagesStabilize}). This step may be viewed as a global version of an important theorem in Hartshorne and Speiser \cite{HaSp} about co-finite modules (over a complete, local, and $F$-finite ring) with a left action of the Frobenius. In fact, the proof in \cite{HaSp} of the counterpart of \autoref{t.ImagesStabilize} implicitly uses Cartier modules, but is much less general than the proof given here which is due to Gabber \cite[Lemma 13.1]{Gabber.tStruc} (see \cite{LyubKazmZheng,BliSchTakZha_DiscRat} for an application of this result to the theory of test ideals, and \cite{Bli.MinimalGamma} for yet another version). In that same paper of Gabber one finds as well the other crucial ingredient; as part of the proof of his \cite[Lemma 13.3]{Gabber.tStruc}:
\begin{enumerate}
\stepcounter{enumi}
\item \label{z.ingrB} (\autoref{t.UniformSupport}) Let $(M,C)$ be a coherent Cartier module with $C(M)=M$. Then there is a dense open set $U \subseteq X$ such that for all Cartier submodules $N \subseteq M$ with the same generic rank as $M$, the quotient $M/N$ is supported in the closed set $Y=X\setminus U$.
\end{enumerate}
The critical point is that the closed subset $Y$ only depends on $(M,C)$, but not on the submodule $N$. With \autoref{z.ingrA} and \autoref{z.ingrB} at hand, the proof of the Main Theorem proceeds roughly as follows: By \autoref{z.ingrA} we may replace the above chain by a nil-isomorphic chain where all structural maps $C_i:M_i \to M_i$ are surjective. The surjectivity is then automatically true for all quotients $M/M_i$. Since the generic ranks of $M_i$ can only drop finitely many times, we can assume that -- after truncating -- the generic ranks are constant. Part \autoref{z.ingrB} then implies that all quotients $M/M_i$ are supported on $Y$. In fact, the reducedness of the support of $M/M_i$ by \autoref{z.ingrA} implies that if $I$ is a sheaf of ideals cutting out $Y$, then $I\cdot(M/M_i) = 0$ and hence $IM\subseteq M_i$ for all $i$. So the stabilization of the chain $\{M_i\}_{i > 0}$ is equivalent to the stabilization of the chain $\{M_i/M'\}_{i>0}$ where $M'$ is the smallest Cartier submodule of $M$ containing $IM$. But the latter may be viewed as a chain of Cartier modules on the lower dimensional scheme $Y$, hence by induction this chain stabilizes.
\end{proof}

There are reasons why one expects such a strong finiteness statement for Cartier modules. The most apparent one, which also explains why Gabber's \cite{Gabber.tStruc} is so extremely helpful in our proof, is as follows: Combining key results of Emerton and Kisin \cite{EmKis.Fcrys} or Böckle and Pink \cite{BoPi.CohomCrys} on a Riemann-Hilbert-type correspondence in positive characteristic with ongoing work of the authors \cite{BliBoe.CartierCrys} (see also \autoref{s.applications}) one expects that the category of Cartier crystals on $X$ is equivalent to the perverse constructible sheaves (for the middle perversity) of $\Primefield_q$-vectorspaces on $X_{\text{\'et}}$. One of the key results in \cite{Gabber.tStruc} is that this latter category is artinian and noetherian.

A more explicit but entirely heuristic reason for the finite length of Cartier crystal is the strong contracting property Cartier linear maps enjoy. This is quite elementary and nicely explained by Anderson \cite{AndersonL} where he studies $L$-functions modulo $p$. He shows that in a finitely generated $R=k[x_1,\ldots,x_n]$-module $M$, equipped with a Cartier-liner map $C$, there is a finite dimensional $k$-subspace $M_c$ such that some power of $C$ maps any element $m \in M$ into $M_c$. Very naively this indicates that -- up to nilpotence (thought this is not quite correct!) -- the \emph{finite dimensional} $k$-vectorspace $M_c$ with $C|_{M_c}$ contains all the information of $(M,C)$. So we should expect that $(M,C)$ has finite length since $M_c$ is finite dimensional.

\subsection*{Consequences, applications and relation to other finiteness results.}

There are some immediate formal consequences of the finite length result for Cartier crystals. Namely, one has a theory of Jordan-Hölder series, hence the notion of length of a Cartier crystal, resp. of quasi-length (\ie~length up to nilpotence) of a Cartier-module. This is explained in \autoref{s.SimpleCartDecompSeries} where we further show a companion of a result of Hochster \cite[(5.1) Theorem]{Hochster.FinitenessFmod} about the finiteness of homomorphisms of Lyubeznik's $F$-finite modules. Our result in this context states:
\begin{thm*}[\autoref{t.finiteHom}]
Let $X$ be $F$-finite and let $\CM,\CN$ be Cartier crystals. Then the module $\Hom_{\Crys}(\CM,\CN)$ is a finite-dimensional vector space over~$\BF_q$.
\end{thm*}
The finiteness of the homomorphism set, together with the finite length, formally implies the finiteness of submodules of any Cartier module (up to nilpotence), see also \cite[(5.2) Corollary]{Hochster.FinitenessFmod}.
\begin{thm*}[\autoref{t.FiniteSubmodules}]
Let $X$ be $F$-finite. Then a coherent Cartier module has, up to nilpotence, only finitely many Cartier submodules
\end{thm*}
As already mentioned, the theory of $\CO_X$-modules equipped with a left action of the Frobenius is much more extensively studied as the right actions we investigate here, and there are some deep known results which are quite similar to our Main Theorem. Examples are the above mentioned result of Hartshorne and Speiser in \cite{HaSp} about co-finite modules with a left Frobenius action, and, most importantly, Lyubeznik's result \cite[Theorem 3.2]{Lyub} about the finite length of objects in his category of $\mathcal{F}$-finite modules over a regular ring (which is essentially of finite type over a regular local ring).

In the regular $F$-finite case, the connection between our and Lyubeznik's result is obtained by tensoring with the inverse of the dualizing sheaf $\omega_X$ to obtain an equivalence of Cartier modules and the category of $\gamma$-sheaves which was recently introduced in \cite{Bli.MinimalGamma}. $\gamma$-sheaves are $\CO_X$-modules with a map $\gamma: M  \to \Fr^*M$, and it was shown that the associated category of $\gamma$-crystals is equivalent to the category of Lyubeznik's $\mathcal{F}$-finite modules \cite{Lyub} (in the affine $F$-finite case), resp. to Emerton and Kisin's category of locally finitely generated unit $\CO_X[F]$-modules \cite{EmKis.Fcrys} (finite type over a $F$-finite field case). Hence our main result yields the following (partial) extension of the main result in \cite[Theorem 3.2]{Lyub}.

\begin{thm*}[\autoref{t.FiniteLenthFGunit}]
Let $X$ be a regular $F$-finite scheme. Then every finitely generated unit $\CO_X[F]$-module (resp. $\mathcal{F}$-finite module in the terminology of \cite{Lyub}) has finite length.
\end{thm*}

The connection to Hartshorne and Speiser's theory of co-finite modules with a left Frobenius action relies on Matlis duality and has been used, at least implicitly, many times before, for an explicit use see \cite{Gabber.tStruc,BliSchTakZha_DiscRat}. What it comes down to is that if $F$ is finite, then the Matlis dual functor commutes with $\Fr_*$, and hence for a complete local and $F$-finite ring $R$ this functor yields an equivalence of categories between coherent Cartier modules over $\Spec R$ and co-finite left $R[F]$-modules. This equivalence preserves nilpotence and hence we recover the analogous result of our main theorem for co-finite left $R[\Fr]$-modules, \cf~\autoref{t.FiniteLengthCofinite}, \cite[Theorem 4.7]{Lyub}.

Besides these important consequences we show how our theory implies \emph{global} structural results which generalize analogous results for modules with a \emph{left} Frobenius action over a \emph{local} ring obtained recently by Enescu and Hochster \cite{EnescuHochster.FrobLocCohom}, and Sharp \cite{Sharp.GradedAnnihil}. Our global version states:
\begin{thm*}[\autoref{t.finitelymanyirred}]
Let $X$ be an $F$-finite scheme, and $(M,C)$ a coherent Cartier module with surjective structural map $C$. Then the set
\[
    \{\supp (M/N) | N \subseteq M \text{ a Cartier submodule}\}
\]
is a finite collection of reduced sub-schemes. It consists of the finite unions of the finitely many irreducible closed sub-schemes contained in it.
\end{thm*}
In the same spirit we show the following result, which generalizes \cite[Corollary 4.17]{EnescuHochster.FrobLocCohom}:
\begin{thm*}[\autoref{t.AntiNilpSplit}]
Let $R$ be a Noetherian and $F$-finite ring. Assume that $R$ is $F$-split, \ie~there is a map $C:\Fr_*R \to R$ splitting the Frobenius.
\begin{enumerate}
\item The Cartier module $(R,C)$ has only finitely many Cartier submodules (\autoref{t.FiniteCartFsplit}).
\item If $R$ is moreover complete local, then the injective hull $E_R$ of the residue field of $R$ has \emph{some} left $R[F]$-module structure such that $E_R$ has only finitely many $R[F]$-submodules.
\item If $(R,\frm)$ is also quasi-Gorenstein, then the top local cohomology module $H^d_\frm(R)$ with its standard left $R[F]$-structure has only finitely many $R[F]$-submodules.
\end{enumerate}
\end{thm*}
Note that in this result we show the finiteness of the actual number of Cartier or of $R[F]$-submodules, and not just the finiteness up to nilpotence which already follows by \autoref{t.FiniteSubmodules} as discussed above.

Along the way, in \autoref{r.CompatiblySplit}, we point out how one may derive the finiteness results in the theory of Frobenius splittings about the finiteness of compatibly split primes of \cite{Schwede.Fadjunction,KumarMehta.CompatiblySplit}.

\subsection*{Structure of the paper}
The theory of Cartier modules is presented in \autoref{s.Cart}. We set up the notation and derive basic results needed in the rest of the article. In particular, we deal with the all important notion of nilpotence for Cartier modules, and show that nilpotent Cartier modules are a Serre-subcategory. This section also includes a treatment of duality for a finite morphism (in a very basic way), a discussion of the main example of a Cartier module (the dualizing sheaf), and a thorough discussion of the already interesting case where $X$ is the spectrum of a field.

In \autoref{s.MinimalCartier} the category of Cartier crystals is introduced, and the existence of a unique (\ie~minimal) representative of a Cartier crystal by a Cartier module is shown, \cf~\autoref{t.ExistenceMinimal}. This result is equivalent (see \autoref{t.minimalGamma}) to the existence of minimal $\gamma$-sheaves recently proved by the first author in \cite{Bli.MinimalGamma}, but the proof in the present article is much simpler. We conclude this section with a Kashiwara type equivalence for Cartier crystals, saying that if $Y \subseteq X$ is a closed subscheme, then there is an equivalence of categories of Cartier crystals on $Y$ and Cartier crystals on $X$ with support in $Y$.

The proof of the main result and some immediate consequences are given in \autoref{s.FiniteLength}. Since the arguments given in \cite[Section 13]{Gabber.tStruc} are rather brief, we decided to include these proofs with many details added in this section. This also makes the present paper quite self contained and, as we hope, accessible to a reader with knowledge of algebraic geometry at the level of~\cite{Hartshorne}.

The final \autoref{s.applications} discusses the connections of the category of Cartier modules to other categories of modules with an action of the Frobenius, and it derives the consequences which we outlined above.

\subsection*{Acknowledgements}
The main theorem of this article was worked out while the first author was attending the \emph{Leopoldina-Symposium in Algebraic and Arithmetic Algebraic Geometry} in Ascona in May 2009, and we thank the funding agencies who made this visit possible. The basic theory of Cartier modules we present here is part of the more extensive project \cite{BliBoe.CartierCrys} of the authors, in particular the equivalences used in \autoref{s.applications} were worked out in \cite{BliBoe.CartierCrys} as a by-product of showing an equivalence of the category of $\tau$-crystals of \cite{BoPi.CohomCrys} and of finitely generated unit $\CO_X[\Fr]$-modules of \cite{EmKis.Fcrys}. We are thankful to Karl Schwede for numerous comments on an earlier draft of this manuscript.

\subsection*{Notation and Conventions}
We fix a prime number $p$ and a power $q=p^k$. Let $\Primefield=\Primefield_q$ be the finite field with $q$ elements. Our schemes $X$ will always be locally Noetherian and separated over $\Primefield$. We denote by $\Fr$ the $q$\th-power Frobenius, \ie~the morphism of schemes $\Fr: X \to X$ which is the identity on the underlying topological space, and the $q$\th-power map on sections of $\CO_X$. For the most part we also assume that $X$ is $\Fr$-finite, \ie~that $\Fr$ is a finite morphism of schemes. It would be desirable to relax this hypothesis and to only assume that $X$ is excellent (to have an open regular locus) or possibly also assuming that $X$ has a dualizing complex. By \cite{Kunz.NoethCharP} an $F$-finite ring is excellent, and in \cite{Gabber.tStruc} it is stated that an $F$-finite Noetherian scheme has a dualizing complex (with proof sketched in the affine case).

\section{Cartier modules}\label{s.Cart}
\begin{defn}
A \emph{Cartier module} on $X$ is a quasi-coherent sheaf $M$ of $\CO_X$-modules equipped with a $q^{-1}$-linear map $\Ca=\Ca_M: M \to M$, that is an additive map satisfying $\Ca(r^qm)=r \Ca(M)$ for $r \in \CO_X$. Equivalently, $C$ is a $\CO_X$-linear map $\Ca: \Fr_*M \to M$ and we use both points of view interchangeably in this paper.
\end{defn}
We refer to $\Ca=\Ca_M$ as the structural map of $M$ and often omit the subscript. If $X=\Spec R$ is affine, denote by $R[\Fr]$ the ring
\[
    R[\Fr]=\frac{R\{ \Fr \}}{R\langle r^q \Fr - \Fr r\ |\ r \in R \rangle}
\]
which is obtained from $R$ by adjoining a non-commutative variable subject to the relation $r^q\Fr = \Fr r$ for all $r \in R$. Then a Cartier module on $X$ is nothing but a \emph{right module} over $R[\Fr]$. A general scheme $X$ may be covered by finitely many affines, gluing the respective rings $R[\Fr]$ one obtains a sheaf of rings $\CO_X[\Fr]$. A Cartier module on $X$ is then a sheaf of \emph{right} $\CO_X[\Fr]$-modules which is quasi-coherent as a sheaf of
$\CO_X$-modules.

A map of Cartier modules is a map $\phi: M \to N$ of the underlying quasi-coherent sheaves such that the diagrams
\[
\xymatrix{
    M \ar[d]_{\Ca_M}\ar[r]^\phi &N \ar[d]^{\Ca_N} && \Fr_*M \ar[d]_{\Ca_M}\ar[r]^{\Fr_*\phi} &\Fr_*N \ar[d]^{\Ca_N}\\
    M \ar[r]^\phi &N && M \ar[r]^\phi &N}
\]
commute. Of course both diagrams are equivalent, the left one being a diagram of additive maps, the right one of $\CO_X$-linear maps. Kernels and cokernels of maps of Cartier modules are the same as for (coherent) sheaves of $\CO_X$-modules with the naturally induced Cartier linear map. The category of (coherent) Cartier modules on $X$ is an abelian category, since it is a category of right modules over the ring $\CO_X[\Fr]$. We denote the category of Cartier modules on $X$ by $\Cart(X)$ and the coherent ones by $\CohCart(X)$.

\subsection{Basic pushforward and pullback}
\begin{lem}
\label{t.CartierPushforward}
Let $f: Y \to X$ be a morphism of schemes, then the pushforward $f_*$ for quasi-coherent $\CO_Y$-modules induces a functor $f_*:\Cart(Y) \to \Cart(X)$.
\end{lem}
\begin{proof} Clear since $f_* \circ \Fr^e_*$ is naturally isomorphic to $\Fr^e_* \circ f_*$. \end{proof}
In the cases where $f_*$ preserves coherence for $\CO_Y$-modules, it also preserves coherence for Cartier modules. In particular, if $f$ is finite, then $f_*$ preserves coherence, hence induces a functor from \emph{coherent} Cartier modules on $Y$ to \emph{coherent} Cartier modules on $X$.

Now we discuss pullback for Cartier modules. For simplicity we only consider the special cases of open and closed immersions, leaving a more general discussion for later, see also \cite{BliBoe.CartierCrys}.
\begin{lem}
\label{t.CartierLocalizes}
If $S$ is a sheaf of multiplicative subsets of $\CO_X$, then the localization functor $S^{-1}$ on quasi-coherent $\CO_X$-modules restricts to an exact functor on Cartier modules which preserves coherence. Furthermore, the localization map is a map of Cartier modules.
\end{lem}
\begin{proof}
Note that for any quasi-coherent $\CO_X$-module $M$ we have $\Fr_*(S^{-1}M)=S^{-1}\Fr_*(M)$ since $S^{-q}M = S^{-1}M$ (with the obvious meaning of $S^{-q}$ being the inverse of the multiplicative system consisting of $q$th powers of $S$). This implies that for a Cartier module $M$ the structural map $\Fr_*M \to[\Ca] M$ naturally equips $S^{-1}M$ with a Cartier structure $\Fr_*S^{-1}M \cong S^{-1}\Fr_*M \to[S^{-1}\Ca] M$. Concretely on a fraction $\frac{m}{s}=\frac{ms^{q-1}}{s^q}$ the structural map is given by $\Ca_{S^{-1}M}(\frac{m}{s})=\frac{\Ca_M(ms^{q-1})}{s}$. Obviously the localization map $m \mapsto \frac{m}{1}$ is Cartier linear.
\end{proof}

\begin{cor}
\label{t.RestrOpen}
If $i: U \into X$ is an open immersion then the restriction $i^*$ for quasi-coherent $\CO_X$-modules induces a natural functor from Cartier modules on $X$ to Cartier modules on $U$ that preserves coherence.
\end{cor}
\begin{proof}
Take local generators $(f_1,\ldots,f_k)$ of an ideal defining the complement $X-U$ of $U$. Then $M|_U$ is the kernel of the fist map in the \Cech~complex on the $U_i=\Spec \CO_X[f^{-1}_i]$. Each $M|_{U_i}$ has a natural structure of a Cartier module by \autoref{t.CartierLocalizes}, hence so do the entries of the \Cech~complex. Since the maps in the \Cech~complex are (sums of) localization maps, which are Cartier linear by \autoref{t.CartierLocalizes}, we conclude that the kernel $M|_U=i^*M$ is naturally a Cartier module.
\end{proof}

\begin{prop}
\label{t.ibDefinition}
Let $i: Y \into X$ be a closed subscheme defined by the sheaf of ideals $I$. Let $(M,\Ca)$ be a Cartier module on $X$. Then
\[
    i^\flat(M) \defeq \SHom_{\CO_X}(i_*\CO_Y,M) = M[I] \defeq \{m \in M | Im = 0 \} \text{ (viewed as an $\CO_Y$-module)}
\]
is naturally a Cartier module on $Y$. $i^\flat$ defines a left exact functor from $\Cart(Y)$ to $\Cart(X)$ which preserves coherence.
\end{prop}
\begin{proof}
Everything follows once we observe that $\Ca(M[I]) \subseteq M[I]$: If $Im=0$, then $I\Ca(m)=\Ca(I^{[q]}m) \subseteq \Ca(Im) =0$.
\end{proof}

\begin{prop}
\label{t.i*ibEquiv}
Let $i: Y \into X$ be a closed subscheme defined by the sheaf of ideals $I$. Let $M$ (resp.~$N$) a Cartier module on $X$ (resp.~$Y$). Then
\[
    i_*i^\flat M  \cong M[I] \text{ and } i^\flat i_*N  \cong N\, .
\]
\end{prop}
\begin{proof}
The first equality is clear from the definition, and the second follows since $I i_*N=0$ and hence $(i_*N)[I] = i_*N$.
\end{proof}

\subsection{Nilpotence}
\label{s.nilpotence}
In our treatment of Cartier modules the notion of nilpotence is critical. It means that some power of the structural map acts as zero.
\begin{defn}
\label{d.nilopetence}
A Cartier module $(M,\Ca)$ is called \emph{nilpotent} if $\Ca^e(M)=0$ for some $M$. The smallest $e$ such that $\Ca^e(M)=0$ is called the \emph{order of nilpotence} of $M$ and denoted by $\nilord(M)$. A Cartier module is called \emph{locally nilpotent} if it is the union of its nilpotent Cartier submodules.
\end{defn}
One might expect that local nilpotence for a Cartier module $M$ is equivalent to the annihilation of every section of $M$ by some power of $C$. As the example below shows, this is not the case. However, local nilpotence is equivalent to requesting that for every local section $m \in \CM (U)$ there is $e >0$ such that $\Ca^e(\CO_X(U) \cdot m)=0$, \ie~the whole $\CO_X$-submodule generated by $m$ has to be nilpotent.\footnote{To check the equivalence, let $M_e$ be the subsheaf whose sections $M_e(U)$ are precisely the sections of $M(U)$ such that $\Ca^e(\CO_X(U)\cdot m)=0$. One immediately verifies that each $M_e$ is a Cartier submodule of $M$, which is nilpotent, and any nilpotent Cartier submodule $N$ of $M$, say with $\nilord(M)=e$, is contained in $M_e$; see also the proof of \autoref{t.nil-part}.} The discrepancy between the annihilation of a section and of the sub-sheaf generated by this section is explained by the observation that the kernel of $\Ca$ is generally not an $\CO_X$-submodule of $M$; it is only a $\CO_X$-submodule of $\Fr_*M$. Hence, a Cartier module might have no nilpotent submodules even if $\Ca$ has a nontrivial kernel:
\begin{ex}
Let $\Ca:\Primefield[x] \to[x^i \mapsto x^{((i+1)/q)-1)}] \Primefield[x]$ (where $x^{((i+1)/q)-1)}=0$ whenever the exponent is not integral). Then $\Primefield[x]$ has no submodule on which $\Ca$ is zero and hence none on which $C$ is nilpotent, even though every element of $\Primefield[x]$ is annihilated by some power of $\Ca$.
\end{ex}
In this article we are mainly concerned with \emph{coherent} Cartier modules. In this case the notion of nilpotence and local-nilpotence agree.
\begin{lem}
\label{t.nil-is-locnil-for-coherent}
If $(M,\Ca)$ is coherent then $M$ is nilpotent if and only if $M$ is locally nilpotent.
\end{lem}
\begin{proof}
If a Cartier sheaf is nilpotent, it clearly is locally nilpotent. For the converse, observe that if $N_1$ and $N_2$ are nilpotent Cartier subsheaves of $M$ of nilpotency oders $e_1$ and $e_2$, respectively, then $N_1+N_2$ is nilpotent of order $\max\{e_1,e_2\}$. Using the noetherianess of $M$, it now easily follows that locally nilpotent implies nilpotent.
\end{proof}
\begin{lem}
\label{t.nil_localizes}
A coherent Cartier module $M$ is nilpotent if and only if for all (closed) points $x \in X$ the localization $M_x$ is nilpotent.
\end{lem}
\begin{proof}
If $M$ is nilpotent then so are all of its localizations. Conversely, if for $x \in X$ the map $C^e$ is zero, then $C^e$ is zero on some open neighborhood of $x$. Covering $X$ by finitely many such neighborhoods (for various (closed) points $x \in X$) and taking the maximum of the appearing $e$'s it follows that $C^e(M)=0$ as claimed.
\end{proof}
\begin{lem}
\label{t.NilSerreSub}
The category of nilpotent Cartier modules on $X$ is a Serre subcategory of the category of coherent Cartier modules on $X$, \ie~it is an abelian subcategory which is closed under extensions and subquotients.
\end{lem}
\begin{proof}
Everything is clear except the closedness under extensions. For this we consider a short exact sequence $0\to N' \to N \to[\pi] N''\to0$ and assume that both $N'$ and $N''$ are nilpotent, \ie, there are  $e$ and $f$ such that $\Ca_{N'}^{e}=\Ca_{N''}^{f}=0$. Given an arbitrary $n \in N$ we have $\pi(C^f(n))=C^f(\pi(n))=0$ since $N''$ is nilpotent of order $\leq f$. Hence, $C^f(n)$ is in the kernel of $\pi$ which is equal to $N'$. Since $N'$ is nilpotent of order $e$ it follows that $C^e(C^f(n))=0$. Since $e$ and $f$ do not depend on $n \in N$ we have $C^{e+f}(N)=0$, so $N$ is nilpotent.
\end{proof}
The proof shows that $\max\{\nilord(N'),\nilord(N'')\} \leq \nilord(N) \leq \nilord(N') + \nilord(N'')$.

The following simple observation shows that each Cartier module has a largest (locally) nilpotent submodule.
\begin{prop}
\label{t.nil-part}
Let $(M,\Ca)$ be a coherent Cartier module. Then there exists a unique Cartier submodule $M_{\nil}$ such that
\begin{enumerate}
\item $M_{\nil}$ is nilpotent.
\item $\overline{M} \defeq M/M_{\nil}$ contains no non-zero nilpotent Cartier submodule.
\end{enumerate}
\end{prop}
\begin{proof}
Define $M_{\nil}$ as the sum over all locally nilpotent Cartier sub-sheaves of $M$. By the proof of \autoref{t.nil-is-locnil-for-coherent} this is a locally nilpotent Cartier subsheaf of $M$. Since $M$ is noetherian, \autoref{t.nil-is-locnil-for-coherent} asserts that $M_{\nil}$ is nilpotent.  Suppose now that $N/M_{\nil}$ is a nilpotent Cartier subsheaf of $M/M_{\nil}$ for some $N\supset M_{\nil}$. Then by \autoref{t.NilSerreSub}, the Cartier submodule $N$ of $M$ is nilpotent. By the definition of $M_{\nil}$ we have $N\subset M_{\nil}$. Hence $M_{\nil}$ satisfies both properties asserted in the proposition.
\end{proof}

\begin{rem}
If the scheme $X$ is $F$-finite, it is shown in \cite{BliBoe.CartierCrys} that one obtains analogs of the above statements for \emph{locally nilpotent} Cartier modules. Namely we show that locally nilpotent Cartier modules form a Serre subcategory, and that a quasi-coherent Cartier module $M$ has a maximal locally nilpotent submodule $M_{\nil}$ such that the quotient $M/M_{\nil}$ has no nilpotent submodules. One can construct an example which shows that both statements fail if $X$ is not $F$-finite. These generalizations are extremely useful in the construction of certain functors on the category of Cartier modules carried out in \cite{BliBoe.CartierCrys}. However, since we do not need this added generality here, we only included the basic nilpotent case.
\end{rem}

The following observation is at the heart of what is to follow. It roughly says that up to nilpotence one can replace a coherent Cartier module by one with a surjective structural map. For a Cartier module, to have a surjective structural map is equivalent to not having nilpotent Cartier quotients. If $X=\Spec R$ is affine and $R$ is local and $\Fr$-finite, this statement is dual to a condition in \cite{HaSp} on the nilpotence of locally nilpotent co-finite modules with a \emph{left} Frobenius action. A global version of it in the regular $\Fr$-finite case can be found in \cite{Bli.MinimalGamma}. The following slick proof is due to Gabber \cite{Gabber.tStruc}. Note that we do not assume $X$ to be $\Fr$-finite.
\begin{prop}
\label{t.ImagesStabilize}
Let $(M,\Ca)$ be a coherent Cartier module on $X$. Then the descending sequence of images $\Ca^i(M)$
\[
    M \supseteq \Ca(M) \supseteq \Ca^2(M) \supseteq \ldots
\]
stabilizes.
\end{prop}
\begin{proof}
We can and will assume that $X=\Spec R$ is affine. Note that $\Ca^i(M)$ is not only an $R$-submodule of $M$ (since $r\Ca^i(m)=\Ca^i(r^q m)$), but even a coherent Cartier submodule of $M$. Furthermore, $\Ca(S^{-1}M)=\Ca(S^{-q}M)=S^{-1}\Ca(M)$ hence the formation of the image of the structural map $\Ca$ commutes with localization. Let
\[
    Y_n=\supp(\Ca^n(M)/\Ca^{n+1}(M))=\{ x \in \Spec R | \Ca(\Ca^{n-1}(M))_x \neq \Ca^{n-1}(M_x) \}
\]
be the closed subset of $X$ where the $n$\th step of the chain is \emph{not} an equality. The second equality in the displayed equation uses the fact that the formation of the image of $C$ commutes with localization as explained above. If $x \not\in Y_n$, then $\Ca(\Ca^{n-1}(M_x)) = \Ca^{n-1}(M_x)$ and applying $C$ yields $\Ca(\Ca^{n}(M_x)) = \Ca^n(M_x)$ such that $x \not\in Y_{n+1}$. Hence $\{Y_n\}_{n\geq 0}$ forms a descending sequence of closed subsets of $X$. Since $X$ is Noetherian, and hence compact, this sequence must stabilize. Hence there exists $n\geq 0$ such that for all $m \geq n$ we have $Y_n=Y_m(\defeq Y)$. By replacing $M$ by $\Ca^n(M)$ we may assume that for all $n$ we have $Y=\supp(\Ca^n(M)/\Ca^{n+1}(M))$.

The statement that the chain $\Ca^n(M)$ stabilizes means precisely that $Y$ is empty. So let us assume otherwise and let $\frm$ be a generic point of $Y$ (\ie~the generic point of an irreducible component $Z$ of $Y$). Localizing at $\frm$ (and replacing $X$ by $\Spec R_{\frm}$ and $M$ by $M_{\frm}$) we may hence assume that $X=\Spec R$, where $(R,\frm)$ is a local ring and $\supp(\Ca^n(M)/\Ca^{n+1}(M))=\{\frm\}$ for all $n$. In particular, for $n=0$ we get that there is $k > 0$ such that $\frm^k M \subseteq C(M)$. Then for $x \in \frm^k$
\[
    x^2M \subseteq x\frm^k M \subseteq x \Ca(M) = \Ca(x^pM) \subseteq \Ca(x^2M)
\]
and by iterating we get for all $x\in \frm^k$ and all $a \in \mathbb{N}$ that $x^2M \subseteq \Ca^a(M)$. Hence
\[
    \frm^{k(b+1)}M \subseteq (\frm^k)^{[2]}M \subseteq \Ca^a(M) \text{ for all } a \in \mathbb{N}
\]
where $b$ is the number of generators of $\frm^k$ and $(\frm^k)^{[2]}$ is the ideal generated by the squares of the elements of $\frm^k$ (and it is easy to check that there is an inclusion $\frm^{k(b+1)} \subseteq (\frm^k)^{[2]}$). Hence the chain $\Ca^a(M)$ stabilizes if and only if the chain $\Ca^a(M)/\frm^{k(b+1)}M$ does. But this is a chain in $M/\frm^{k(b+1)}M$, which has finite length. This contradicts our assumption that $\supp(\Ca^n(M)/\Ca^{n+1}(M)) \neq \emptyset$ for all~$n$.
\end{proof}

\begin{cor}
\label{t.QuotSurjStruct}
Let $(M,\Ca)$ be a coherent Cartier module. Then there is a unique Cartier subsheaf $(\underline{M},\underline{C})$ such that
\begin{enumerate}
\item the quotient $M/\underline{M}$ is nilpotent, and
\item the structural map $\underline{\Ca}: \underline{M} \to \underline{M}$ is surjective (equiv. $\underline{M}$ has no non-zero nilpotent quotient).
\end{enumerate}
\end{cor}
\begin{proof}
The stable image $\underline{M}\defeq \Ca^e(M)$ for $e \gg 0$ which exists by \autoref{t.ImagesStabilize} has all the desired properties.
\end{proof}
\begin{rem}
The canonically assigned modules $\underline{M}$ and $\overline{M}$ may also be characterized by the following conditions, more resembling a universal property.
\begin{enumerate}
\item $\underline{M}$ is the smallest Cartier submodule $N$ of $M$ such that the quotient $M/N$ is nilpotent.
\item $\overline{M}$ is the smallest  Cartier quotient $N$ of $M$ such that the kernel of $M \onto N$ is nilpotent.
\end{enumerate}
\end{rem}

\subsection{Duality for finite morphisms}
\label{s.dualityFinite}
In this section we briefly recall some parts of the duality theory for finite morphisms as explained in \cite[Chapter III.6]{HartshorneRD}. We want to apply this in particular to the Frobenius morphism $\Fr$, hence we shall assume now that $\Fr$ is a finite map, \ie~we assume that $X$ is an $\Fr$-finite scheme. One consequence of this will be an interpretation of the Cartier module structure, \ie~the map $\Ca: \Fr_*M \to M$, via its adjoint under the duality of the finite Frobenius map, \ie~a map $\kappa: M \to \Fr^\flat M \cong \SHom_{\CO_X}(\Fr_*\CO_X,M)$. Here, and from here onward, we denote by $\SHom_{\CO_X}(M,N)$ the sheaf of local homomorphisms from $M$ to $N$, \ie~the sheaf associated to the presheaf
$$U\mapsto 
\Hom_{\CO_X(U)}(M(U),N(U)).$$

Let $f: Y \to X$ be a finite morphism of (locally Noetherian) schemes. Then the functor
\[
    f^\flat: \CO_X\text{--modules} \to \CO_Y\text{--modules}
\]
is defined by $f^\flat M= f^{-1}\SHom_{\CO_X}(f_*\CO_Y,M)$ viewed as an $\CO_Y$-module. Locally, for $Y=\Spec S$ and $X=\Spec R$, both affine, it is just given by $f^\flat M=\SHom_R(S,M)$ viewed as an $S$-module via its action on the first entry of the $\Hom$. As one easily checks, in the case of a closed embedding $i: Y \into X$ this definition agrees with $i^\flat$ given above in \autoref{t.ibDefinition}. In the case that $f: Y \to X$ is \'etale, one has that $f^\flat \cong f^*$. Furthermore $(\usc)^\flat$ is compatible with composition.

The duality for a finite morphism \cite[Theorem III.6.7]{HartshorneRD} states that the functor $f^\flat$ is a right adjoint to the functor $f_*$. More precisely
\begin{thm}[Duality for a finite morphism]
\label{t.dualityfinite}
Let $f: Y \to X$ be a finite morphism of locally Noetherian schemes. Then the trace map $\Tr_f: f_*f^\flat M \to M$ (given by evaluation at $1$) induces an isomorphism
\[
    f_*\SHom_{\CO_Y}(M,f^\flat N) \to[\cong] \SHom_{\CO_X}(f_*M,N)
\]
for every quasi-coherent $\CO_Y$-module $M$, and $\CO_X$-module $N$.
\end{thm}
In the affine case of a finite homomorphism of rings $R \to S$, an $R$-module $N$, and $S$-module $M$ (where everything can be reduced to) this comes down to nothing more than the well-known isomorphism
\[
    \Hom_S(M,\Hom_R(S,N)) \to[\cong] \Hom_R(M,N)
\]
given by sending $\phi$ to the map $m \mapsto (\phi(m))(1)$ (whose inverse is the map sending $\psi$ to the map $m \mapsto (s \mapsto  \psi(sm))$. From this we obtain the following characterization of Cartier modules:
\begin{prop}
\label{t.CartAdjoint}
Let $X$ be $\Fr$-finite and $M$ a quasi-coherent $\CO_X$-module. Then a Cartier module structure on $X$ is equivalently given by one of the following:
\begin{enumerate}
    \item A right $\CO_X[\Fr]$-module structure on $M$, compatible with the $\CO_X$-structure.
    \item An $\CO_X$-linear map $\Ca: \Fr_*M \to M$.
    \item An $\CO_X$-linear map $\Cadj: M \to \Fr^\flat M$.
\end{enumerate}
\end{prop}
The equivalence of the first two items was already observed above, and the equivalence of the second and third item is the just discussed adjointness. One easily (but tediously) verifies that the adjoint map to $\Ca^i$ is the map $\kappa^i$ defined inductively via $\kappa^1=\kappa$ and $\kappa^i=\Fr^\flat(\kappa^{i-1})\circ \kappa = \Fr^{(i-1)\flat}\kappa \circ \kappa^{i-1}$. As a corollary we get:
\begin{cor}
\label{t.nilkappa}
Let $X$ be $\Fr$-finite, and $(M,\Ca,\kappa)$ a Cartier module on $X$. Then the kernel of $\kappa^i$ is the maximal nilpotent Cartier submodule of order $\leq i$, and $M_{\nil}=\bigcup_i \ker \kappa^i$.
\end{cor}
\begin{proof}
Clearly, the $K_i=\ker \kappa^i_M$ form an increasing sequence of $\CO_X$-submodule. Since $\kappa(K_i) \subseteq \Fr^\flat(K_{i-1}) \subseteq \Fr^\flat(K_i)$, it follows that $K_i$ is a Cartier submodule of $M$. So $K_i$ is clearly the largest Cartier submodule of $M$ on which $\kappa^i$ acts as zero. But $\kappa_i$ being zero is equivalent to $\Ca^i$ being zero since $\kappa^i$ and $\Ca^i$ are adjoint morphisms.
\end{proof}

\subsection{Dualizing complexes and the Cartier isomorphism}
\label{s.dualizingsheaf}

We recall some parts of Residues and Duality \cite{HartshorneRD} on the existence and uniqueness of dualizing complexes. A \emph{dualizing complex} $\omega_X^\bullet$ is a bounded complex of finite injective dimension such that for all complexes of sheaves with coherent cohomology $M^\bullet$ on $X$, the natural double-dualizing map
\[
    M^\bullet \to \RSHom_{\CO_X}(\RSHom_{\CO_X}(M^\bullet,\omega_X^\bullet),\omega_X^\bullet)
\]
is an isomorphism. It is shown in \cite[Prop.~V.2.1]{HartshorneRD}  that this condition on $\omega_X^\bullet$ is equivalent to the validity of the same condition for all {\em bounded} complexes of sheaves with coherent cohomology $M^\bullet$. We only consider noetherian schemes of finite dimension. Then dualizing complexes exist in many circumstances, \cf~\cite[V.10]{HartshorneRD}:x
\begin{enumerate}
\item If $X$ is regular, or more generally if $X$ is Gorenstein, then $\CO_X$ itself is a dualizing complex.
\item If $X$ is essentially of finite type over a scheme that has a dualizing complex, then $X$ has a dualizing complex.
\item In particular, if $X$ is of finite type over a field, or over a Gorenstein local ring, then $X$ has a dualizing complex.
\end{enumerate}
As one readily verifies from the definition of a dualizing complex, if $\omega_X^\bullet$ is dualizing, then so is $\omega_X^\bullet \tensor \mathcal{L}[n]$ for any integer $n$ and any invertible sheaf $\mathcal{L}$ on $X$. By \cite[Theorem V.3.1]{HartshorneRD} this is all that can happen, \ie~any two dualizing complexes on $X$ differ only by a shift and tensorization with an invertible sheaf. If $X$ is Cohen-Macaulay, then a dualizing complex has cohomology concentrated in a single degree. More generally, if $X$ is normal, then we denote by the \emph{dualizing sheaf} $\omega_X$ the unique reflexive sheaf on $X$ which agrees with the dualizing complex $\omega_X^\bullet$ on the Cohen-Macaulay locus. Hence, given two dualizing complexes $\omega^\bullet_X$ and $\omega^{\prime\bullet}_X$, on sufficiently small Zariski open subsets $U \subseteq X$ (trivializing $\mathcal{L}$), the restrictions of $\omega^\bullet_X$ and $\omega^{\prime\bullet}_X[n]$ to $U$ are isomorphic. Since in a normal scheme the non-Cohen-Macaulay locus has codimension $\geq 3$, this isomorphism extends to an isomorphism of the corresponding dualizing sheaves. In particular, if $X$ is the spectrum of a local ring, then dualizing complexes are unique (up to shift), and so are dualizing sheaves.

One is interested in the behavior of dualizing complexes under pullback by a (finite) morphism $f: Y \to X$. More precisely, we consider the functor
\[
    f^!M^\bullet = \RSHom_{\CO_X}(f_*\CO_Y,M^\bullet)
\]
on the derived category of complexes of sheaves on $X$ with bounded coherent cohomology.\footnote{generally, this functor is denoted by $f^\flat$, however, we have reserved $f^\flat$ for the zeroth cohomology of this functor throughout the rest of this paper, see \autoref{s.dualityFinite}} By \cite[Proposition V.2.24]{HartshorneRD}, if $\omega_X^\bullet$ is a dualizing complex on $X$, then $f^!\omega_X^\bullet$ is dualizing on $Y$, hence if one has fixed dualizing complexes $\omega_X^\bullet$ on $X$, and $\omega_Y^\bullet$ on $Y$, then one has Zariski locally an isomorphism $\omega_Y^\bullet \cong f^!\omega_X^\bullet[n]$.
One of the main points of duality theory in \cite{HartshorneRD} is that in many cases this last statement also holds globally. That is every essentially of finite type scheme $X$ over some fixed scheme $S$ which has a dualizing complex $\omega_S^\bullet$, can be equipped with a dualizing complex $\omega_X^\bullet$ (namely $\omega_X^\bullet=\eta^! \omega_X^\bullet$, where $\eta: X \to S$ is the morphism exhibiting the $S$-structure of $X$), such that if $f: Y \to X$ is a morphism of $S$-schemes, then $f^!(\omega^\bullet_X) \cong \omega^\bullet_Y$. Applying this to the Frobenius we obtain:
\begin{prop}
\label{t.CartierFiniteOverGor}
Let $X$ be $F$-finite scheme over a local Gorenstein scheme $S=\Spec R$. Then there is a dualizing complex $\omega_X^\bullet$ such that $\omega_X^\bullet \cong \Fr^!\omega_X^\bullet$. If $X$ is also normal, one obtains an isomorphism $\omega_X \cong \Fr^\flat \omega_X$ of dualizing sheaves.
\end{prop}
\begin{proof}
Since $X$ is an $S$-scheme we have a morphism $\eta: X \to S$. Since $S$ is Gorenstein local, $R=\CO_S$ itself is a dualizing sheaf on $X$. Let $\omega_X^\bullet = \eta^!R$. Since $R$ is local we have an isomorphism $\Fr_S^!R \cong R$. Consider now the relative Frobenius diagram of $X$ over $S$:
\[
\xymatrix{ X \ar@/^{1pc}/[rrd]^{\Fr_X=\Fr}\ar[rd]^{\Fr_{X/S}}\ar@/_{1pc}/[rdd]_{\eta} \\
&  X' \ar[r]^{\Fr^\prime_S}\ar[d]^{\eta'} & X \ar[d]^{\eta} \\
                                    &      S \ar[r]^{\Fr_S}& S
}
\]
Applying $\eta^{\prime !}$ and combining with the identities $\eta^! \cong \Fr_{X/S}^! \circ \eta^{\prime !}$ and $\eta^{\prime !}\circ \Fr_S^! \cong \Fr^{\prime !}\circ \eta^!$ we get
\[
    \omega_X^\bullet \cong \eta^!R \cong \Fr_{X/S}^!\eta{\prime !}R \cong \Fr_{X/S}^! \eta^{\prime !}\Fr_S^!R \cong \Fr_{X/S}^! \Fr_S^{\prime !}\eta^!R \cong \Fr_X^!\omega_X^\bullet.
\]
Since the non-Cohen-Macaulay locus in a normal scheme has codimension $\geq 3$, the induced isomorphism $\omega_X \cong \Fr^!\omega_X$ on the Cohen-Macaulay locus extends to all of $X$.
\end{proof}
We point out that the proof in fact shows that \autoref{t.CartierFiniteOverGor} holds whenever the base scheme $S$ has a dualizing complex $\omega^\bullet_S$ such that $\Fr^!\omega^\bullet_S \cong \omega^\bullet_S$.

We need to extend this result slightly to a case that is not directly stated in \cite{HartshorneRD}.
\begin{prop}
Let $X$ be $\Fr$-finite and affine. Then $X$ has a dualizing complex $\omega_X^\bullet$. Furthermore, for all sufficiently small affine open subsets $U \subseteq X$, one has (quasi-)isomorphisms of dualizing complexes
\[
    \omega_X^\bullet|_U = \omega_U^\bullet \cong \Fr^!\omega_U^\bullet
\]
which, in the case that $X$ is also normal, induce isomorphisms $\omega_U \cong \Fr^\flat\omega_U$ of dualizing sheaves.
\end{prop}
\begin{proof}
The existence of a dualizing complex for any noetherian, affine, $F$-finite scheme is shown in \cite[Remark 13.6]{Gabber.tStruc}. The local isomorphisms exist due to the uniqueness of dualizing complexes discussed above (the shift is irrelevant since $\Fr$ has relative dimension zero). The final statement about the dualizing sheaves follows formally.
\end{proof}
Summarizing, we see that if $X$ is normal and $\Fr$-finite and either (a) $X$ is essentially of finite type over a Gorenstein local ring, or (b) $X$ is sufficiently affine, then the there is a dualizing sheaf $\omega_X$ such that $\omega_X \cong \Fr^\flat \omega_X$; in other words $\omega_X$ carries the structure of a Cartier module.
\begin{rem}
An interesting question seems to be the existence of a dualizing complex for an arbitrary $F$-finite scheme $X$. Gabber states in \cite[Remark 13.6]{Gabber.tStruc} that this is the case but only provides a proof in the case that $X$ is affine. Even more interestingly one might ask if any $\Fr$-finite scheme $X$ has a dualizing complex such that $\omega_X^\bullet \cong \Fr^! \omega_X^\bullet$ globally. This is not clear to us, even in the case that $X$ is regular. This would of course imply, that on an $\Fr$-finite scheme, there is a dualizing sheaf equipped with a Cartier module structure $\Fr_* \omega_X \to \omega_X$ such that its adjoint $\omega_X \to \Fr^\flat \omega_X$ is an isomorphism.
\end{rem}
We finish by pointing out an explicit example of the above general construction, namely the classical Cartier operator \cite{Cartier57} on a smooth and finite type over a perfect field scheme.

\begin{ex}
For simplicity let $X=\Spec k[x_1,\ldots,x_n]$ and $k=\Primefield_p$. Let
\[
    \Omega_X = k[x_1,\ldots,x_n]\langle dx_1,\ldots,dx_n \rangle
\]
be the module of Kähler differentials and $\Omega^\bullet_X$ the resulting de~Rham complex. The \emph{inverse Cartier operator} is the unique map
\[
    C^{-1} \colon \Omega_X^1 \to H^1(\Fr_*\Omega_X^\bullet)
\]
which for a local section $x \in \CO_X$ sends $dx$ to the class of $x^{p-1}dx$ (or, for easier memorization it sends $x^i \frac{dx}{x} \to x^{pi} \frac{dx}{x}$). The Cartier isomorphism states that the induced map $\Omega_X^i \to H^i(\Fr_*\Omega_X^\bullet)$ is an isomorphism, which is checked by an explicit calculation, see \cite{DI}. In particular, on top differential forms $\omega_X = \Omega^n_X$ one obtains the \emph{Cartier operator}
\[
    C \colon \Fr_*\omega_X \to \omega_X
\]
as the composition of the natural surjection $\Fr_*\omega_X \to H^n(\Fr_*\Omega^\bullet_X)$ and the inverse of the inverse Cartier operator $C^{-1} \colon \omega_X \to[\cong] H^n(\Fr_*\Omega^\bullet_X)$. Explicitly, this map is given by
\[
\textstyle{
    x_1^{k_1}\cdot\ldots\cdot x_n^{k_n}(\frac{dx_1}{x_1}\wedge \ldots \wedge \frac{dx_n}{x_n}) \mapsto x_1^{k_1/p}\cdot\ldots\cdot x_n^{k_n/p}(\frac{dx_1}{x_1}\wedge \ldots \wedge \frac{dx_n}{x_n})}
\]
where we set $x_j^{k_j/p}=0$ if $k_j/p$ is not integral. By \autoref{t.dualityfinite}, the adjoint of the Cartier operator $C: \Fr_* \omega_X \to \omega_X$ is the $\CO_X$--linear map
\[
    \kappa: \omega_X \to \Fr^\flat \omega_X.
\]
It can easily check by hand that $\kappa: \omega_X \to \Fr^! \omega_X$ is an isomorphism.

More generally, for any regular scheme $X$, essentially of finite type over a perfect field $k$, one has, as above, a canonical isomorphism $\Omega_{X}^i \to \mathcal{H}^i(\Fr_*\Omega_{X}^\bullet)$ which induces the \emph{Cartier operator} $C: \Fr_{*}\omega_{X} \mapsto \omega_{X}$. If in addition $X/k$ is normal, and denoting by $\omega_X$ the unique reflexive sheaf on $X$ which agrees with the top differential forms on the regular locus $X_{reg}$, then the Cartier operator on the smooth locus induces a Cartier linear map $C: \Fr_*\omega_X \to \omega_X$.
\end{ex}

\subsection{Cartier modules over $F$-finite fields}
\label{s.CartierFields}

In this section we investigate the case of $X=\Spec k$ where $k$ is a field. That is we study vector spaces over $k$ equipped with a $q^{-1}$-linear endomorphism, or equivalently, with a right action of $k[\Fr]$. The point we make here is that if $[k:k^q]<\infty$ then duality $\Hom_k(\blank,k)$ can be used to reduce the study of \emph{right} $k[\Fr]$ vector-spaces to that of $k$-vectors-spaces with a \emph{left} $k[\Fr]$-action. This is slightly more involved than one might suspect, owing to the just mentioned subtlety concerning the uniqueness of dualizing modules. The problem is, that there is no canonical identification of a $k$-vector-space with its $k$-dual.

We assume that the field $k$ is $\Fr$-finite, that is, $\Fr_*k$ is a finite dimensional $k$-vector-space. Then its $k$-dual $(\Fr_*k)^\vee=\Hom_k(\Fr_*k,k)$ is (non-canonically) isomorphic to $\Fr_*k$ as a $k$-vectorspace. But $\Hom_k(\Fr_*k,k)$ is also a $\Fr_*k$-vector-space via the action on the first factor of the $\Hom$, and, for dimension reasons, necessarily one-dimensional as a $\Fr_*k$-vector-space. As fields, $\Fr_*k$ is simply isomorphic to $k$. We write $\Fr^\flat k$ for the one dimensional $k \cong \Fr_* k$ vector-space $(\Fr_* k)^\vee$. The choice we make now, is the choice of a generator of $\Fr^\flat k$ as a $k$-vector-space, that is the choice of a $k$-isomorphism $\Fr^\flat k \cong k$. Clearly, two such choices will differ by multiplication with an element in $k^\times$. Since all that will follow depends on the choice of an isomorphism $\Fr^\flat k \cong k$, we will fix one from now on. \autoref{t.dualityfinite}  shows that we have made the choice of a Cartier-module structure on $\Fr_*k \to k$ such that its adjoint is the isomorphism $\Fr^\flat k \cong k$. Also note, that in the case where $k$ is perfect, this problem does not arise since then the Frobenius is an isomorphism $k \cong \Fr_*k$ (canonically).
\begin{lem}
Let $k$ be an $\Fr$-finite field, then there is an isomorphism of functors $\Fr_* \circ (\blank)^\vee \cong (\blank)^\vee \circ \Fr_*$ on $k$-vector-spaces.
\end{lem}
\begin{proof}
We just compute for a $k$-vector-space $V$
\[
\Hom_k(\Fr_*V,k) \overset{\ref{t.dualityfinite}}{\cong}  \Fr_*\Hom_k(V,\Fr^\flat k) {\cong} \Fr_*\Hom_k(V,k)
\]
where the first isomorphism is the duality of the finite morphism (in the simple case of a field) in \autoref{t.dualityfinite}, and the second one is induced by our fixed isomorphism $F^\flat k \cong k$. Note that a different choice of such isomorphism only changes everything by multiplication of an element in $k^\times$.
\end{proof}

In the simple case of $k$-vectorspaces considered here we had to fix an isomorphism between the dualizing module $f^\flat k$ and the dualizing module $k$. Once that has been fixed it induces for every finitely generated field extension $i:k \into K$ an isomorphism of the $K$-dualizing modules $(i^\flat k)$ (non-canonically isomorphic to $K$) with $F_K^\flat (i^\flat k)$ via the functoriality of $(\blank)^\flat$. It is in this sense that the constructions that follow are compatible with finitely generated field extensions. In particular for any finite extension field $K$ of $k$ we will from now on denote by $(\blank)^\vee_K$ the functor $\Hom_K(V,i^\flat k)$ and hence get a natural isomorphism $F_{K*}(\blank)^\vee_K \cong (F_{K*}\blank)^\vee_K$ of functors on $K$-vector-spaces.

\begin{prop}\label{t.duality.for.fields}
Let $K/k$ be a finitely generated extension of an $\Fr$-finite field $k$. Then the duality functor $(\blank)^\vee_K$ induces an equivalence between the categories of coherent Cartier modules over $K$ (\ie~finite dimensional $K$-vector-spaces with a \emph{right} action of the Frobenius) and finite dimensional $K$-vector-spaces with a \emph{left} action of the Frobenius. The equivalence $(\blank)^\vee_K$ furthermore preserves nilpotence.
\end{prop}
\begin{proof}
Given a coherent Cartier module $(V,C)$ over $K$, that is a finite dimensional $K$-vector-space $V$ with a $K$-linear map $C: \Fr_*V \to V$. Applying $(\blank)^\vee_K=(\blank)^\vee$ we get
\[
    C^\vee: V^\vee \to[C^\vee] (\Fr_*V)^\vee \cong \Fr_*V^\vee
\]
with the second isomorphism being the transformation of functors from the preceding lemma. Clearly, the map $C^\vee$ gives $V^\vee$ the structure of a $K$-vector-space with a left action of the Frobenius, \ie~a left $K[\Fr]$-module. Conversely, the same observation applies and it is easy to verify that these functors induce inverse equivalences.

The fact that nilpotence is preserved follows immediately from the construction.
\end{proof}

\autoref{t.duality.for.fields} enables us to fall back on the much better studied theory of left $k[\Fr]$ modules concerning questions of right $k[\Fr]$-modules. However, this appears to be only possible in the case that $k$ is $F$-finite, since otherwise $\Fr_*k$ is infinite dimensional over $k$, and hence $(\Fr_*k)^\vee$ is not isomorphic to $\Fr_* k$, and in particular $\Fr^\flat k$ (as defined here) cannot be isomorphic to $k$. This also shows that the assumption of $\Fr$-finiteness we make on our schemes is essential to most methods in this paper and quite likely to many of the results as well.

Since we took great care to set up everything compatibly with respect to finitely generated field extensions, we get the following result.

\begin{prop}
Let $i: k \to K$ be a finitely generated extension of an $\Fr$-finite field $k$. Then there are natural functorial isomorphisms
\begin{align*}
     (i_*V)^{\vee}_k &\cong i_*(V^{\vee}_K) \qquad \qquad (i_*W)^{\vee}_k \cong i_*(W^{\vee}_K)\\
     (i^\flat U)^{\vee}_K &\cong i^*(U^{\vee}_k) \qquad \qquad (i^*H)^{\vee}_K \cong i^\flat (H^{\vee}_k)
\end{align*}
for all finite-dimensional right/left $k[\Fr]$-vector-spaces $V,W$, and right/left $K[\Fr]$-vector-spaces $U,H$.

In other words, the duality functor $(\blank)^{\vee}_K$ induces an equivalence of categories between finite dimensional right and left $k[\Fr]$-vector-spaces which commutes with $i_*$ and interchanges $i^\flat$ and $i^*$.
\end{prop}
\begin{proof}
It is sufficient to show the isomorphisms in the first row as this will imply the ones in the second row since $(i_*,i^!)$ and $(i^*,i_*)$ are adjoint pairs and $(\blank)^\vee$ is a duality. For this just observe that $(i_*V)^{\vee}_k = \Hom_k(i_*V,k) \cong i_*\Hom_k(V,i^\flat k) = i_*(V^{\vee}_K)$ where the last equality is due to our convention concerning duality. The second isomorphism follows analogously.
\end{proof}

With this at hand, we can restate some facts from the theory of left $k[\Fr]$-modules in terms of Cartier modules over $k$.

\begin{prop}
\label{t.CartierFieldBasics}
Let $k$ be $F$-finite, and $V,W$ be Cartier modules. Then
\begin{enumerate}
\item There is a direct sum decomposition $V=V_{\nil} \oplus \underline{V}$ where $V_{\nil}$ is nilpotent and $\underline{V}$ does not have a nilpotent subspace.
\item Let $V=\underline{V}$, then there is a finite separable extension $K$ of $k$ such that $V_K=\Hom_k(K,V)\cong K \tensor_k V$ is isomorphic to $\oplus_n \omega_K$ where $\omega_K$ denotes the Cartier module $i^\flat k \to[\cong] F^\flat(i^\flat k)$ corresponding to the standard Frobenius action on $K$ via the duality $(\blank)^\vee_K$.
\item The fixed points of $C_V$ form a $\Primefield_q$-vector-space of dimension less or equal to $\dim_k \underline{V}$ with equality after a finite separable extension.
\item If $V=\underline{V}$ and $W=\underline{W}$, then $\Hom_{\Cart(k)}(V,W)$ is a finite dimensional vector-space over $\Primefield_q$.
\item If $V=\underline{V}$, then there are only finitely many Cartier submodules of $V$.
\end{enumerate}
\end{prop}
\begin{proof}
Since $\dim_k V = \dim_k \Fr^\flat V$, the structural map $\kappa: V \to \Fr^\flat V$ is injective if and only if it is surjective, if and only if it is bijective, and each instance is equivalent to $V$ not having a nilpotent subspace. For a finite field extension, so in particular for $\Fr: k \to k$, the functor $(\blank)^\flat$ is exact, and it follows that the surjectivity of $C$ is in fact equivalent to the surjectivity of its adjoint $\kappa$. Now let $V_{\nil}=\ker(\kappa^i)$ and $\underline{V}=C^i(V)$ for $i$ sufficiently large (as $k$-subspaces of a finite dimensional $k$-vectorspace these will stabilize, \cf~\autoref{s.nilpotence}). Then it is easy to see by dimension reasons that $V_{\nil}$ and $\underline{V}$ have complementary dimension and zero intersection, which shows the first part. The second part follows from the well-known corresponding statement for left $k[\Fr]$-modules using the preceding definition, see \cite{DieudonLie}. The finial three items follow again by the just proven duality from \cite[(4.2) Theorem]{Hochster.FinitenessFmod}
\end{proof}

\begin{rem}
The duality in the case of vector spaces we work out here is a very special case of the duality we describe in \cite{BliBoe.CartierCrys} where we show that Grothendieck-Serre duality for quasi-coherent $\CO_X$-modules can be extended to a duality between right and left $\CO_X[\Fr]$-modules. We decided to include the simple special case of fields here as an instructive example (already making apparent some of the difficulties in duality theory), and because it will be useful later in this paper, \cf~\autoref{t.simpleFiniteHom}.
\end{rem}

\section{Minimal Cartier modules and Cartier crystals}
\label{s.MinimalCartier}
In this section we consider Cartier modules \emph{up to nil-isomorphism}, leading to notion of a Cartier crystal. The main result in this section is the existence of minimal Cartier modules. This result allows one to pick in each nil-isomorphism class of Cartier modules (\ie~for each Cartier crystal) a canonical (minimal) representative. This minimal representative is characterized by having neither nilpotent submodules nor nilpotent quotient modules. In this section we do not assume that our schemes are $\Fr$-finite.

\subsection{Cartier Crystals}
\begin{defn}
A map of Cartier modules is called a \emph{nil-isomorphism} if kernel and cokernel are nilpotent.
\end{defn}
\begin{lem}
The composition of nil-isomorphisms is again a nil-isomorphism.
\end{lem}
\begin{proof}
Let $M' \to[\phi] M \to[\psi] M''$ be a composition of nil-isomorphisms. The kernel of $\phi$ and $\psi$ are both nilpotent, say of order less or equal to $e$. Then it is easy to see that $\ker (\psi \circ \phi)$ is nilpotent of order less or equal to $2e$. Similarly for the cokernel of $(\psi \circ \phi)$.
\end{proof}
\begin{lem}
A map of coherent Cartier modules $M \to[\phi] N$ is a nil-isomorphism if and only if for all $x \in X$ the induced map $\phi_x: M_x \to N_x$ is a nil-isomorphism.
\end{lem}
\begin{proof}
The only if part is clear and the rest follows from \autoref{t.nil_localizes}.
\end{proof}

\begin{defn}
The category of \emph{Cartier crystals} is the abelian category obtained from the category of coherent Cartier modules by localization with respect to nil-isomorphisms.
\end{defn}
This definition means, that the objects are just Cartier modules, but a homomorphism from $M$ to $N$ is a diagram (left fraction):
\[
    M \Longleftarrow M' \to N
\]
where $M' \Longrightarrow M$ is a nil-isomorphism (which we will denote for convenience by $\Longrightarrow$), and $M' \to N$ is a morphism of Cartier modules. The general theory of localization of an abelian category at a multiplicative system defined from a Serre subcategory yields that Cartier crystals form also an abelian category, \cf~\cite{BoPi.CohomCrys} as a reference for standard facts about localization of categories. We will denote the Cartier crystals associated with a Cartier module $M$ by the corresponding calligraphic letter $\CM$. Likewise, if $\CM$ is a crystal, $M$ denotes \emph{some} Cartier module giving rise to that crystal. Some immediate properties are:
\begin{lem}
\label{t.CrysNilp}
Let $M$, $N$ be two coherent Cartier modules and $\CM$, $\CN$ their associated crystals.
\begin{enumerate}
\item A map of Cartier modules $M \to N$ is a nil-isomorphism if and only if the induced map $\CM \to \CN$ of crystals is an isomorphism of Cartier crystals.
\item $\CM \cong \CN$ are isomorphic as crystals if and only if there is a Cartier module $M'$ and nil-isomorphisms $M \Longleftarrow M' \Longrightarrow N$ of Cartier modules.
\item $\CM \cong 0$ as a crystal if and only if $M$ is nilpotent.
\end{enumerate}
\end{lem}

\begin{prop}
\label{t.basicFuncForCrys}
The previously defined functors $i^\flat$ (for closed embeddings), $j^*$ (open embeddings), and $f_*$ (in general) preserve nil-isomorphism, and hence induce well-defined functors on $\Crys$ whenever they preserve coherence.\footnote{Coherence is clearly preserved under $j^*$ and $i^\flat$, but generally not under $f_*$.}
\end{prop}
\begin{proof}
Since the functor $i^\flat(\blank)=\usc[I]$ is left exact, it preserves nil-injections. Now let $N \into M$ an injective map such that $M/N$ is nilpotent. Since $N[I]=M[I]\cap M$, it follows that the cokernel of $N[I] \into M[I]$ is a submodule of the nilpotent $M/N$ and hence also nilpotent. Hence $i^\flat$ preserves nil-isomorphisms.

Since $j^*$ is exact for open embeddings $j: U \to X$, it preserves nil-isomorphisms since it preserves nilpotence. The case of $f_*$ follows also easily using that $\Fr^e_*$ is exact, since $\Fr$ is affine.
\end{proof}

\subsection{Minimal Cartier modules}
In this section we show that for a Cartier crystal $\CM$ one can choose a representing Cartier module in a functorial fashion.
\begin{defn}\label{d.minimal}
A Cartier module $M$ is called \emph{minimal} if the following two conditions are satisfied:
\begin{enumerate}
\item \label{d.minimal.sub}$M$ has no nilpotent Cartier submodules. (equiv. (in the $\Fr$-finite case) the adjoint structural map $\kappa: M \to \Fr^\flat M$ is injective)
\item \label{d.minimal.quot}$M$ has no nilpotent Cartier quotients. (equiv. the structural map $\Ca: \Fr_*M \to M$ is surjective)
\end{enumerate}
\end{defn}

\begin{lem}
\label{t.NilIsoMinimalIsIso}
A nil-isomorphism of minimal Cartier modules is an isomorphism.
\end{lem}
\begin{proof}
Let $\phi: M \to N$ be a nil-isomorphism of Cartier modules. If $M$ satisfies \autoref{d.minimal} \autoref{d.minimal.sub} then $\phi$ cannot have a nilpotent kernel, hence $\phi$ is injective. If $N$ satisfies \autoref{d.minimal} \autoref{d.minimal.quot}, then $\phi$ cannot have nilpotent cokernel, so $\phi$ must be surjective.
\end{proof}

\begin{lem}
\label{t.MinimalLocalizes}
If $M$ is minimal, then any localization of $M$ is minimal.
\end{lem}
\begin{proof}
Since localization is exact, the structural map on any localization is also surjective. Since any nilpotent submodule of a localization of $M$ has nonzero intersection (pre-image under the localization map) with $M$ it follows that a localization cannot have nilpotent submodules, unless $M$ did.
\end{proof}

\begin{thm}
\label{t.ExistenceMinimal}
The assignment $M \mapsto M_{min}$ which maps a coherent Cartier module $M$ to the nil-isomorphic Cartier module $M_{min}$ defines a functor. The Cartier module $M_{min}$ can be obtained as a sub-quotient of $M$, and any minimal Cartier module $N$ which is nil-isomorphic to $M$ is isomorphic to $M_{min}$.
\end{thm}
\begin{proof}
Define
\[
    M_{min} \defeq \underline{(\overline{M})}
\]
where $\overline{\blank}$ and $\underline{\blank}$ are as in \autoref{t.nil-part} and \autoref{t.ImagesStabilize}. Clearly $M_{min}$ is minimal, a sub-quotient of $M$, and nil-isomorphic to $M$. Since $\overline{\blank}$ and $\underline{\blank}$ are both functorial, so is their composition. If $N \Longleftarrow N' \Longrightarrow M$ is another sequence of nil-isomorphisms with $N$ minimal, the functoriality implies that we have nil-isomorphisms $N=N_{min} \Longleftarrow N'_{min} \Longrightarrow M_{min}$. Since nil-isomorphisms of minimal Cartier modules are isomorphisms by \autoref{t.NilIsoMinimalIsIso} it follows that $N \cong M_{min}$ which implies uniqueness.
\end{proof}
Note that the uniqueness part of \autoref{t.ExistenceMinimal} shows that the order in which $\overline{\blank}$ and $\underline{\blank}$ are applied to arrive at a minimal Cartier module is irrelevant.
\begin{lem}
\label{t.HomCrysMinEqual}
If $M,N$ are minimal Cartier sheaves and $\CM,\CN$ denote their associated crystals, then $\Hom_{\Cart}(M,N)=\Hom_{\Crys}(\CM,\CN)$.
\end{lem}
\begin{proof}
By \autoref{t.CrysNilp} the map $\Hom_{\Cart}(M,N) \to \Hom_{\Crys}(\CM,\CN)$ induced by the functor $M \mapsto \CM$ is injective. For surjectivity, let $\phi \in \Hom_{\Crys}(\CM,\CN)$ be represented by a map of Cartier modules
\[ M \Longleftarrow M' \to[\phi'] N. \]
The functoriality of $(\blank)_{min}$ induces a map $M \cong M'_{min} \to[\phi'_{min}] N$ which also represents $\phi$, but is a honest map of Cartier modules from $M \to N$.
\end{proof}

\begin{thm}
The category of Cartier crystals is equivalent to the category of minimal Cartier modules.
\end{thm}
\begin{proof}
One has the natural functors
\[
    \MinCart(X) \into \Cart(X) \to \Crys(X)\, .
\]
Conversely, the map that assigns to each Cartier crystal $\CM$ represented by some Cartier module $M$ the minimal Cartier module $M_{min}$ is a well defined functor from Cartier crystals to minimal Cartier modules. Well defined-ness follows since if $M'$ also represents $\CM$, then $M$ and $M'$ are nil-isomorphic and hence $M_{min} \cong M'_{min}$. Now the preceding lemma shows that these two functors are inverse equivalences.
\end{proof}

\subsection{Crystalline support and Kashiwara equivalence}
A simple, however crucial, property of Cartier modules is captured by the following observation
\begin{lem}
\label{t.AnnRadical}
Let $(M,\Ca)$ be a Cartier module with \emph{surjective} structural map, \ie~$C(M)=M$. Then $\Ann_{\CO_X} M$ is a sheaf of radical ideals.
\end{lem}
\begin{proof}
Let $x \in \CO_X$ be a local section such that $x^kM=0$. Then for $q^e \geq k$ we have $0=x^kM=x^{q^e}M=\Ca^e(x^{q^e}M)=xC^e(M)=xM$.
\end{proof}

\begin{cor}
\label{thm.finitelentghlenth1}
If $X=\Spec R$ is affine, $\frp$ is a prime ideal, and $M$ a Cartier module with \emph{surjective} structural map, then $\frp^k M=0$ implies $\frp M=0$.
\end{cor}

\begin{defn}
Let $M$ be a coherent Cartier module on $X$. We define $\suppcrys(M)$, the \emph{crystalline support} of $M$, as the set of all $x \in X$ such that $M_x$ is not nilpotent.
\end{defn}
\begin{lem}
\label{l.nilSuppWellDef}
\begin{enumerate}
\item \label{z.crysSuppA}If the structural map of $M$ is surjective, then $\suppcrys(M) = \supp(M)$.
\item If $M$ and $N$ are nil-isomorphic, then $\suppcrys(M) = \suppcrys(N)$, In particular, using \autoref{z.crysSuppA} we have $\suppcrys(M)=\supp(\underline{M})$.
\item \label{z.crysSuppC} $\suppcrys(M)$ is a closed reduced subscheme.
\end{enumerate}
\end{lem}
\begin{proof}
Since surjectivity of the structural map passes to localization, every $M_x$ for $x \in X$ has this property. But clearly, a Cartier module with surjective structural map is nilpotent if and only it is zero, proving the first part. Similarly, if $M$ and $N$ are nil-isomorphic, then so are $M_x$ and $N_x$ for all $x \in X$. Hence $M_x$ is nilpotent if and only if $N_x$ is nilpotent. Part \autoref{z.crysSuppA} shows that $\suppcrys(M)$ is a closed set, equal to $\supp(\underline{M})$. But also the scheme structure given by $\Ann_{\CO_X} \underline{M}$ is reduced since by \autoref{t.AnnRadical} $\Ann_{\CO_X} \underline{M}$ is radical, since $\underline{M}$ has surjective structural map.
\end{proof}
These observations guarantee the well defined-ness of the following notion of support for a Cartier crystal.
\begin{defn}
Let $\CM$ be a Cartier crystal represented by the Cartier module $M$. Then the \emph{support} of $\CM$ is defined as $\supp(\CM)=\suppcrys(M)$, which by \autoref{z.crysSuppC} above is a closed and reduced subscheme.
\end{defn}

Let $Y \subseteq X$ be a closed embedding, then the Kashiwara-type equivalence we will derive states that the category of Cartier crystals $\Crys(Y)$ on $Y$ is equivalent to the category of Cartier crystals on $X$ which are supported in $Y$. This latter subcategory of $\Crys(X)$ we denote by $\Crys_Y(X)$.
\begin{lem}
Let $Y \subseteq X$ be a closed embedding and let $(M,\Ca)$ be a coherent Cartier module on $X$ with surjective structural map. Then $M$ is supported in $Y$, if and only if $M=M[I]=\{m \in M | Im=0 \}$.
\end{lem}
\begin{proof}
Let $I$ be the sheaf of ideals defining $Y \subseteq X$. Then a coherent $M$ is supported in $Y$ if and only if some power $I^n$ annihilates $M$, \ie~$I^nM=0$. Since we assumed that the structural map of $M$ is surjective this is, by \autoref{t.AnnRadical}, equivalent to $IM=0$, and hence $M[I]=M$.
\end{proof}
\begin{prop}
\label{t.Kashiwara}
Let $i:Y \to X$ be a closed embedding. Then the functors $i_*$ and $i^\flat$ induce an equivalence of categories of Cartier crystals $\Crys(Y)$ on $Y$ and $\Crys_Y(X)$, the Cartier crystals on $X$ with support in $Y$.
\end{prop}
\begin{proof}
In \autoref{t.i*ibEquiv} we have seen that $i^\flat i_* = \id$ and $i_*i^\flat M = M[I]$ on Cartier modules. But the two functors on Crystals are defined by evaluating them on representing Cartier modules, \cf~\autoref{t.basicFuncForCrys}. We may hence choose the Cartier module $M$ representing a crystal $\CM$ to have surjective structural map, by replacing $M$ by $\underline{M}$. Then the support of $\CM$ is equal to the support of $M$, and now the preceding lemma shows $i_*i^\flat \cong \id$ on Cartier crystals.
\end{proof}

\begin{rem}
It is useful to point out that the property of minimality for coherent Cartier modules is preserved by the local cohomology functor $H^0_I(\blank)$. Furthermore, if $M$ has surjective structural map, so in particular if $M$ is minimal, \autoref{t.AnnRadical} shows that $H^0_I(M)=M[I]$. To check minimality of $H^0_I(M)$ we only need to show that $C(H^0_I(M))=H^0_I(M)$. But if $I^n m= 0$, and $m=C^e(m')$ for some $m' \in M$ (since, by minimality $C^e(M)=M$) then $C^e((I^n)^{[q^e]})=I^n C^e(m')=I^n m= 0$. This means that the $\CO_X$-submodule $(I^n)^{[q^e]}m'$ of $M$ is killed by $C^e$. Again by minimality of $M$ this implies that $(I^n)^{[q^e]}m'=0$, hence $m' \in H^0_I(M)$. This implies that for a closed immersion $i: Y \to X$ the functor $i^\flat$ preserves minimality. Furthermore, since $i$ is affine, $i_*$ is exact and it follows that $i_*$ preserves minimality. Hence the preceding equivalence for Cartier crystals may also be thought of as an explicit equivalence on minimal coherent Cartier modules.
\end{rem}

\section{Finite length for Cartier crystals}
\label{s.FiniteLength}

In this section we prove our main result about coherent Cartier modules. Namely that every Cartier module has \emph{up to nilpotence} finite length. From this result we will later derive a series of consequences, recovering and extending other finiteness-results of modules under the presence of an action of the Frobenius.

The following result of Gabber is the key technical ingredient to showing that, up to nilpotence, coherent Cartier modules have finite length.

\begin{prop}[\protect{\cite[Lemma 13.2]{Gabber.tStruc}}]
\label{thm.opennonilquot}
Let $X$ be irreducible and $\Fr$-finite, and let $M$ be a coherent Cartier module on $X$. Then there is an open subset $U \subseteq X$ such that for all $x \in U$ non generic\footnote{we call $x$ non-generic in $U$ if $x$ is not the generic point of an irreducible component of $U$.} we have:
\begin{equation}\tag{$\star$}
\label{cond.finitelengthnil}
    \text{All finite length Cartier quotients of $M_x$ are nilpotent.}
\end{equation}
\end{prop}
\begin{proof}
We may replace $M$ by $\underline{M}$ and assume that the structural map of $M$ is surjective.\footnote{For this, one has to observe that \autoref{cond.finitelengthnil} for $\underline{M}$ at a point $x$ implies \autoref{cond.finitelengthnil} for $M$ at $x$. To ease notation we temporarily replace $M$ by $M_x$. If $M/N$ is a finite length quotient, then the submodule $\underline{M}/(N \cap \underline{M})$ has finite length. By \autoref{cond.finitelengthnil} for $\underline{M}$ it is hence nilpotent. As a quotient of the nilpotent $M/\underline{M}$, $M/(\underline{M}+N)$ is nilpotent. But then, since nilpotence is preserved under extensions by \autoref{t.NilSerreSub}, it follows that $M/N$ is nilpotent as desired.} The surjectivity of the structural map guarantees that the nil-radical $I$ of $\CO_X$ annihilates $M$ by \autoref{t.AnnRadical}, since $I^n=0$ for $n \gg 0$. Hence we may replace $X$ by $X_{red}$ and assume that $X$ is reduced, and hence integral.

We may further replace $M$ by $\overline{M}=M/M_{\nil}$ and hence assume that there are no nilpotent Cartier submodules.\footnote{For this, one has to check that for a point $x \in X$ condition \autoref{cond.finitelengthnil} for $\overline{M}_x$ implies \autoref{cond.finitelengthnil} for $M_x$. To simplify notation we may temporarily replace $M$ by $M_x$. Take $N \subseteq M$ such that the quotient $M/N$ has finite length. Then $M/(M_{\nil}+N)=\overline{M}/(N/(M_{\nil}\cap N))$ has finite length, and is hence nilpotent by assumption that \autoref{cond.finitelengthnil} holds for $\overline{M}$. But clearly, $M_{\nil} \cap N$ is nilpotent as well, and hence so is $M/N$ as an extension of nilpotents by \autoref{t.NilSerreSub}.} Since we assume that $X$ is $F$-finite, this condition means that the adjoint to the structural map
\[
    M \to[\kappa] \Fr^\flat M = \SHom_{\CO_X}(\Fr_*\CO_X,M)
\]
is injective, \cf~\autoref{t.CartAdjoint} and \autoref{t.nilkappa}. Let $U=X_{reg}$ be the dense open (since $X$ is reduced and $\Fr$-finite implies excellent) regular locus of $M$. By a result of Kunz \cite{Kunz}, $\Fr$ is finite flat on $U$. Replacing $X$ by $U$, $\Fr_*\CO_X$ is a locally free $\CO_X$-module of finite rank. Hence,  $\Fr^\flat=\SHom_{\CO_X}(\Fr_*\CO_X,\blank)$ is an exact functor. Further shrinking $U$ we may assume that $M$ is (locally) free of rank $r$, and hence so is $\Fr^\flat M$. Shrinking $U$ once again we may assume that $\kappa$ is surjective, hence an isomorphism (we already arranged that $\kappa$ is injective by killing $M_{\nil}$). For $x \in U$ not the generic point of $U$, let $M_x \onto N$ be a Cartier module quotient of finite length on $\Spec \CO_{X,x}$. Since $x$ is not the generic point, $\CO_{X,x}=:(R,\frm)$ is a regular local ring of dimension $\geq 1$. The exactness of $\Fr^\flat$ implies that $N \to \Fr^{\flat}N$ is still surjective, hence $\length{N} \geq \length(\Fr^\flat N)$. But since $\dim R \geq 1$, $\Fr^\flat = \Hom_R(\Fr_*R,\blank)$ multiplies the length of finite length modules by a power of $q$.\footnote{Since $\Fr^\flat$ is exact it is enough to show that $\length(\Fr^\flat(R/\frm)) = \length(R/\frm^{[q]})$. For this we note that $\Hom_R(\Fr_*R,R/\frm)=\Hom_{R/\frm}(\Fr_*(R/\frm^{[q]}),R/\frm)$ since for $x \in \frm$ we have $0=x \phi(r) = \phi(x^q r)$ for all $\phi \in \Hom_R(\Fr_*R,R/\frm)$. But $R/\frm^{[q]}$ is free of dimension equal to $\length(R/\frm^{[q]})$ over $R/\frm$, hence so is $\Hom_{R/\frm}(\Fr_*(R/\frm^{[q]}),R/\frm)$ over $\Fr_*(R/m)$. This shows the claim.} But this implies that $N=0$, so, in particular, $N$ is nilpotent as desired.
\end{proof}
\begin{rem}
The $F$-finiteness was crucially used in this proof when we employed the adjointness of $\Fr^\flat$ and $\Fr_*$ for the finite Frobenius $\Fr$. However, we expect that the result itself holds much more generally. At least when $R$ is local, or essentially of finite type over a regular local ring. Faithfully flatness of completion can be used to reduce the local case to the complete local case. Then our Kashiwara equivalence \autoref{t.Kashiwara} reduces to the regular case.
\end{rem}

\begin{cor}
\label{t.surjnilquotzero}
Let $X$ be $\Fr$-finite and $(M,\Ca)$ a coherent Cartier module with surjective structural map. Then there is a dense open subset $U \subseteq X$ such that for each non-generic $x \in U$, the localization $M_x$ does not have finite length Cartier quotients.
\end{cor}

\begin{proof}
Let $X_1,\ldots,X_n$ be the irreducible components of $X$. On each irreducible open subset $X_i'=X-\bigcup_{j\neq i} X_j$ of $X$ consider the Cartier module $M|_{X'_i}$ (\autoref{t.RestrOpen}). Let $U_i \subseteq X_i'$ be an open subset as in \autoref{thm.opennonilquot}, and let $U=\bigcup U_i$. Since $\Ca$ is surjective on $M$ the same is true for $M_x$ for all $x \in X$ and also for all quotients $M_x/N$. By construction of $U$, for all $x \in U$ the finite length quotients of $M_x$ are nilpotent and hence zero.
\end{proof}

\begin{cor}
\label{t.UniformSupport}
Let $X$ be $\Fr$-finite and $(M,\Ca)$ a coherent Cartier module with surjective structural map. There is a closed subset $Y = V(I) \subseteq X$, not containing an irreducible component of $X$, such that for all Cartier module quotients $M/N$ of $M$ whose support $\supp M/N$ does not contain an irreducible component of $X$, one has $\supp M/N \subseteq Y$. In fact $\Ann_{\CO_X} M/N \supseteq \sqrt{I} \supseteq I$.
\end{cor}
\begin{proof}
Let $U$ be as in the preceding corollary and let $Y=X-U$. Let $Z=V(\frp)$ be an irreducible component of $\supp M/N$. By assumption $\frp$ is not the generic point of an irreducible component of  $U$ (=generic point of irreducible component of $X$). The localization $(M/N)_{\frp}=M_{\frp}/N_{\frp}$ is a non-zero and finite length module over $\CO_{X,\frp}$. Since we assumed that the structural map of $M$, and hence of $M_{\frp}/N_{\frp}$ is surjective, it cannot be nilpotent, hence $\frp \not\in U$ (since by assumption $\frp$ is not the generic points of an irreducible component of $U$ which would be the only other option). It follows that $Z \subseteq Y$, hence $\supp M/N \subseteq Y$. Again, since $M/N$ has surjective structural map, the ideal $\Ann M/N$ is radical by \autoref{t.AnnRadical}, hence $I \subseteq \sqrt{I} \subseteq \Ann M/N$.
\end{proof}

As a corollary of this result one derives by induction the following.
\begin{prop}[\cite{Gabber.tStruc}]
\label{t.GabberFinite}
Let $X$ be $\Fr$-finite and $M$ a coherent Cartier module. There is a finite subset $S \subseteq X$ such that for all $x \in X \setminus S$
\begin{equation*}\tag{$\star$}
    \text{the finite length quotients of $M_x$ are nilpotent.}
\end{equation*}
\end{prop}
\begin{proof}
As before we may replace $M$ by $\underline{M}$ and assume that $\Ca$ is surjective.
We first observe that if $Y \subseteq X$ is a closed irreducible subset given by a sheaf of ideals $I$, then for all $x \in Y$ the Cartier module $M$ on $X$ satisfies \autoref{cond.finitelengthnil} at $x$ if the Cartier module $\widetilde{M}=M/\sum_{i\geq 0} \Ca^i(IM)$ on $Y$ satisfies \autoref{cond.finitelengthnil}: To see this let $N \subseteq M_x$ be such that $M_x/N$ has finite length. Since $\Ca$ is surjective, this implies that $\frm_x M_x \subseteq N$ by \autoref{thm.finitelentghlenth1}. Since $x \in Y$, and hence $I_x \subseteq \frm_x$ we have $I_xM_x \subseteq N$. It follows that $\sum_{i \geq 0} \Ca^i(I_xM_x) \subseteq N$ since the former is the smallest Cartier submodule of $M_x$ that contains $I_xM_x$. If we denote by $\widetilde{N}$ the image of $N$ in $\widetilde{M}_x$, it can be checked (using the snake lemma, for example) that the natural map $M_x/N \to \widetilde{M}_x/\widetilde{N}$ is an isomorphism. But \autoref{cond.finitelengthnil} for $\widetilde{M}$ on $Y$ implies that the latter is nilpotent.

To conclude the argument we can now use induction on dimension. By considering each irreducible component of $X$ separately we may reduce to $X$ irreducible. Then \autoref{thm.opennonilquot} yields an open set $U$ (minus the generic point) where \autoref{cond.finitelengthnil} holds for $M$. By induction we know that the conclusion of the theorem holds for $\widetilde M$ on each of the finitely many irreducible components of the complement of $U$. This finishes the argument.
\end{proof}

Gabber uses the preceding proposition in \cite{Gabber.tStruc} as a crucial ingredient in his proof that the category of perverse constructible sheaves of $\BF_q$-vectorspaces $\Perv_c(X_{et},\BF_q) \subseteq D^b_c(X_{et},\BF_q)$ is Noetherian and Artinian. Via Grothendieck-Serre duality and using \cite{BoPi.CohomCrys} or \cite{EmKis.Fcrys}, and \cite{BliBoe.CartierCrys} the category of Cartier crystals is expected to be equivalent to $\Perv_c(X_{et},\BF_q)$. This, in turn would imply the finite length for Cartier crystals. The purpose of this section is to give a direct argument for this expected finite length of Cartier crystals:

\begin{thm}
\label{t.FiniteLenth}
Let $X$ be a scheme satisfying \autoref{thm.opennonilquot} (e.g. $X$ is $\Fr$-finite) and $M$ a coherent Cartier module. Then any descending chain
\[
    M \supseteq M_1 \supseteq M_2 \supseteq M_3 \supseteq \ldots
\]
of Cartier submodules of $M$ stabilizes up to nilpotence. This means that for $i \gg 0$ the quotients $M_i/M_{i+1}$ are nilpotent.
\end{thm}
\begin{proof}
We may replace the given chain by the nil-isomorphic one
\[
    \underline{M} \supseteq \underline{M_1} \supseteq \underline{M_2} \supseteq \underline{M_3} \supseteq \ldots
\]
by \autoref{t.ImagesStabilize} -- since clearly -- $M_i/M_j$ is nilpotent if and only if $\underline{M_i}=\underline{M_j}$. Hence we may assume that for all $i$ we have $\Ca(M_i)=M_i$.
One proceeds by induction on the dimension of $X$, the case of dimension zero being obvious. There are only finitely many steps in this chain where the generic rank on some irreducible component of $X$ can drop. By truncating we may hence assume that on each irreducible component the generic rank is constant in the descending chain. This implies, that for each $i$ the support $\supp M/M_i$ does not contain any irreducible component of $X$. Applying \autoref{t.UniformSupport} we have $Y = V(I) \subseteq X$ of strictly smaller dimension than $X$ such that \emph{for all $i$} the support $\supp M/M_i \subseteq Y$. In fact we even have $IM \subseteq M_i$ for all $i$. This implies that $M' \defeq \sum_{i\geq 0} \Ca^i(IM)$ (which is the smallest Cartier submodule of $M$ containing $IM$) is also contained in each $M_i$, since each $M_i$ is a Cartier submodule containing $IM$. But this implies that the original chain stabilizes if and only if the chain
\[
    M/M' \supseteq M_1/M' \supseteq M_2/M' \supseteq \ldots
\]
stabilizes. The latter, however, is a chain of Cartier submodules of $M/M'$ on $Y$ (since $IM \subseteq M'$) with the property that $\Ca_{M/M'}(M_i/M')=M_i/M'$ for all $i$. By induction we know this chain stabilizes.
\end{proof}
As an immediate corollary we obtain, in the language of Cartier crystal, our main result.
\begin{cor}
\label{t.CrysFiniteLength}
Let $X$ be a scheme satisfying \autoref{thm.opennonilquot} (e.g. $X$ is $\Fr$-finite), then every coherent Cartier crystal is Noetherian and Artinian, \ie~ascending, as well as descending chains of Cartier sub-crystals are eventually constant.
\end{cor}
\begin{proof}
That ascending chains stabilize already follows from the Noetherianness of $X$. The stabilization of descending chains is precisely the statement of \autoref{t.FiniteLenth}.
\end{proof}

As a first application we derive global versions of analogs of results of Enescu and Hochster \cite{EnescuHochster.FrobLocCohom} and Sharp \cite[Corollary 3.11]{Sharp.GradedAnnihil}. The following (and also its proof) is analogous to \cite[Proposition 3.5]{EnescuHochster.FrobLocCohom}
\begin{lem}
Let $M$ be a Cartier module on $X$. Then
\[
   \CS:= \{Y = \suppcrys M/N \,|\, N\text{ a Cartier submodule of }M \}
\]
is a collection of reduced subschemes, closed under taking irreducible components and finite unions.
\end{lem}
\begin{proof}
By \autoref{t.AnnRadical} the elements of $\CS$ are reduced subschemes. That $\CS$ is closed under taking finite unions is a consequence of the equality $\Ann_{\CO_X} M/(N\cap N')=\Ann_{\CO_X} M/N \cap \Ann_{\CO_X} M/N'$. It remains to show that the irreducible components  of every $Y = \suppcrys M/N \in \CS$ lie again in $\CS$. For this, we may replace $(X,M)$ by $(Y,\underline{M/N})$. Then we need to show that the irreducible components of $X$ lie in $\CS$ provided that $X=\supp(M)=\suppcrys(M)$. Let $Z$ be such a component, let $Z'$ be the union of those components of $X$ different from $Z$ and let $I$ and $I'$ denote the corresponding ideal sheaves. Since $X$ is reduced, $I$ and $I'$ are radical. Then $I\cap I'=0$ by their definition, and so $\Ann_{\CO_X}(I')=I$. \autoref{t.i*ibEquiv} shows that $M[I']$ is a Cartier subsheaf of $M$. Using $\Ann_{\CO_X}(M)=0$, we compute $\Ann_{\CO_X}(M/M[I'])=\{f\in\CO_X\mid fM\subset M[I']\}=\{f\in\CO_X\mid fI'M=0\}= \{f\in\CO_X\mid fI'=0\}=\Ann_{\CO_X}(I')=I$. Since $I$ is the ideal sheaf of $Z$, the proof is complete.

\end{proof}

\begin{prop}
\label{t.finitelymanyirred}
Let $X$ be $\Fr$-finite and $M$ a coherent Cartier module. Then the collection
\[
    \{Y = \suppcrys M/N \,|\, N\text{ a Cartier submodule of }M \}
\]
is a finite collection of reduced subschemes of $X$. In fact, it consists of the finite unions of the finitely many irreducible subschemes in the collection.
\end{prop}
\begin{proof}
By the preceding Lemma it is enough to show that there are only finitely many irreducible subschemes in this collection. As before, we may replace $M$ by $\underline{M}$ so that the structural map is surjective. If $V=\overline{x}$ for $x \in X$ is an irreducible component of $\supp(M/N)$ then $(M/N)_x$ is a non-zero finite length quotient of $M_x$ (over $\CO_{X,x}$). Since the structural map of $M_x$, and hence of $(M/N)_x$, is surjective, $(M/N)_x$ cannot be nilpotent. It follows that $x$ is an element of the finite set $S$ of \autoref{t.GabberFinite}.
\end{proof}

\begin{rem}
\label{r.EnHoFiniteness}
Enescu and Hochster show in \cite[Theorem 3.6]{EnescuHochster.FrobLocCohom} (see also \cite[Corollary 3.11]{Sharp.GradedAnnihil}) that if $R$ is \emph{local} and $W$ is an Artinian left $R[\Fr]$-module, then the collection
\[
    \{ \Ann_R V | V \subseteq W \text{ an $R[\Fr]$-submodule}\}
\]
is a finite set of radical ideals, consisting of all intersections of the finitely many primes in it. We will explain in \autoref{s.applications} that the precise connection to our result is via Matlis duality. In fact, in the local case, one may give an alternative proof of \autoref{t.finitelymanyirred} using \cite[Theorem 3.1]{EnescuHochster.FrobLocCohom}.
\end{rem}
\begin{rem}
Note that the collection in \autoref{t.finitelymanyirred} is not closed under scheme-theoretic intersection in general.  See, however, the following remark.
\end{rem}
\begin{rem}[Compatibly split subvarieties]
\label{r.CompatiblySplit}
From the viewpoint of Frobenius splittings (see \cite{BrionKumar.FrobSplit}) \autoref{t.finitelymanyirred} may also be interpreted as a generalization of the finiteness of compatibly split subvarieties of an $\Fr$-split variety $X$ obtained by Schwede \cite{Schwede.Fadjunction} (see also \cite{KumarMehta.CompatiblySplit} for a very short proof): A scheme $X$ is called $\Fr$-split if there is a map $\Ca: \Fr^e_*\CO_X \to \CO_X$ splitting the Frobenius. Due to the splitting property, the coherent Cartier module $(\CO_X,\Ca)$ has surjective structural map $\Ca$. A subvariety $Y$ cut out by a sheaf of ideals $I$ is called compatibly split, if $\Ca(I) \subseteq I$, \ie~if and only if the defining ideal $I$ of $Y$  is a Cartier submodule of $\CO_X$, or equivalently, the quotient $\CO_Y$ is a Cartier module quotient of $\CO_X$. Since $\CO_Y$ is a Cartier module quotient, it has also surjective structural map. In particular, by \autoref{t.AnnRadical}, $Y$ is reduced. Since $Y = \supp_{\CO_X} \CO_Y$ the finiteness of compatibly split subschemes $Y$, follows from \autoref{t.finitelymanyirred}. Note that in this case, the collection is, in fact, closed under scheme-theoretic intersection, since $\Ann_{\CO_X} (\CO_X/(I+J))=I+J$.
\end{rem}

\subsection{Finiteness of homomorphisms of Cartier crystals}
\label{s.SimpleCartDecompSeries}
We show that the $\Hom$-sets in the category of Cartier crystals are finite sets. An analog of this result for Lyubeznik's $\mathcal{F}$-finite modules has been obtained by Hochster \cite{Hochster.FinitenessFmod}.
\begin{prop}
\label{t.MinimaSimple}
A non-zero coherent Cartier module is simple if either
\begin{enumerate}
\item $M$ is minimal and simple, or
\item $M$ has zero structural map and is simple as an $\CO_X$-module, \ie~isomorphic to $\CO_X/\frm$ for some maximal ideal $\frm$.
\end{enumerate}
\end{prop}
\begin{proof}
Let $M$ be simple. If $M_{min}$ is non-zero, then $M \cong M_{min}$ since $M_{min}$ is a sub-quotient of $M$. Otherwise $M_{min}=0$, hence $M$ is nilpotent which is the same as $\Ca^n(M)=0$ for some $n$. This implies that $\Ca$ cannot be surjective, implying that $\Ca(M)$ is a \emph{proper} Cartier submodule of $M$. Hence $\Ca(M)=0$ by simplicity of $M$. But if $\Ca$ acts as zero on $M$, then $M$ is simple as a Cartier module if and only if $M$ is simple as an $\CO_X$-module.
\end{proof}

\begin{prop}
A simple Cartier module $M$ has a unique associated prime.
\end{prop}
\begin{proof}
    Let $\frp$ be an associated prime of $M$. Then $M[\frp] \subseteq M$ is clearly a non-zero Cartier submodule. Since $M$ is simple we must have $M[\frp]=M$ which shows our assertion.
\end{proof}

\begin{prop}
\label{t.simpleFiniteHom}
Suppose $X$ is $\Fr$-finite and $M$ is a simple coherent Cartier sheaf on $X$.
\begin{enumerate}
\item If $M$ is minimal and simple, then $\End_{\Cart}(M)$ is a finite field containing $\Primefield_q$. 
\item Otherwise $\Ann_{\CO_X}(M)=:\frm$ is a maximal ideal sheaf, $C_M=0$ and $\End_{\Cart}(M)$ is isomorphic to the residue field at~$\frm$.
\end{enumerate}
\end{prop}

\begin{proof}
In view of \autoref{t.MinimaSimple}, the proof of (b) is obvious. To prove (a), let $\frp$ denote the unique associated prime of $M$. Since $M\subset M_\frp$ and $\End_{\Cart}(M)\subset\End_{\Cart}(M_\frp)$, we may localize at $\frp$ and thus assume that $X=\Spec R$ for $R$ a local ring with maximal ideal $\frp=\Ann_R(M)$ (the condition $C(M)=M$ persists under localization). Therefore $M$ may be regarded as a finite dimensional vector space over the residue field $k:=R/\frp$, and so we are reduced to $R=k$. Since $X$ is $F$-finite, the same holds for $k$. From \autoref{t.CartierFieldBasics} we deduce that $\End_{\Cart}(M)$ is a finite dimensional vector space over $\Primefield_q$ and in particular it is finite. As we also assume that $M$ is simple, the lemma of Schur implies that $\End_{\Cart}(M)$ is a skew field. Now (a) follows since every finite skew field is a field.
\end{proof}
\begin{cor}\label{t.finiteEndCrys}
If $\CM$ is a simple (non-zero) Cartier crystal on $X$, then $\End_{\Crys}(\CM)$ is a finite field containing~$\Primefield_q$.
\end{cor}

Combining \autoref{t.simpleFiniteHom} with the finite length of a Cartier crystal proved in \autoref{t.CrysFiniteLength} we obtain the following analog of the main result in \cite{Hochster.FinitenessFmod} -- and in fact, we follow his arguments closely.

\begin{thm}
\label{t.finiteHom}
Let $X$ be an $\Fr$-finite scheme, and suppose that $M$ and $N$ are coherent Cartier sheaves with $C(M)=M$ and $N_{\nil}=0$. Then $\Hom_{\Cart}(M,N)$ is finite dimensional over $\Primefield_q$, hence a finite set.
\end{thm}
\begin{proof}
The condition on $M$ ensures that its image under any homomorphsm to $N$ lies in $\underline{N}$, the condition on $N$ that $M_{\nil}$ is in the kernel of any such homomorphism. Hence we may replace $M$ by $\overline{M}=M^{\min}$ and $N$ by $\underline{N}=N^{\min}$, so that $\Hom_{\Cart}(M,N) = \Hom_{\Crys}(\CM,\CN)$ if $\CM$ and $\CN$ denote the associated crystals by \autoref{t.HomCrysMinEqual}. Since by \autoref{t.CrysFiniteLength}, both $\CM$ and $\CN$ have finite length (equal to the quasi-length of $M,N$), we can now proceed by induction on $\length(\CM)+\length(\CN)$, the base case of $\CM$ and $\CN$ being simple was just treated in \autoref{t.simpleFiniteHom}. Now take any $0 \neq \CN' \varsubsetneq \CN$ and consider the long exact sequence for $\Hom_{\Crys}$:
\[
    0 \to \Hom_{\Crys}(\CM,\CN') \to \Hom_{\Crys}(\CM,\CN) \to \Hom_{\Crys}(\CM,\CN/\CN') \to \ldots
\]
By induction hypothesis, both ends are finite, so the middle term is finite as well. Proceeding similarly for $\CM$, the result follows.
\end{proof}
\begin{cor}
\label{t.FiniteSubcrystal}
Let $X$ be an $\Fr$-finite scheme. Then any Cartier crystal $\CM$ on $X$ contains only finitely many subcrystals.
\end{cor}
\begin{proof}
By \autoref{t.CrysFiniteLength} and \autoref{t.finiteEndCrys}, the category of Cartier crystals is an artinian and noetherian abelian category in which the endomorphism ring of any simple object is finite. For any such it is well-known that any object has up to isomorphism only finitely many subobjects. Lacking a suitable reference, we include a short proof:
Suppose there is a counterexample $M$. Since every object has finite length, we take $M$ to be the counterexample of smallest length. By the Jordan-H\"older theorem $M$ contains only finitely many simple subobjects. Since any subobject of $M$ contains a simple one, there must be a simple subobject $S$ of $M$ for which there are infintely many distinct subobjects of $M$ containing $S$. But then $M/S$ contains infinitely many distinct subobjects. Since the length of $M/S$ is smaller than that of $M$, we have reached a contradiction.
\end{proof}
The following result allows one to obtain an analogous result for Cartier modules:
\begin{prop}\label{t.subCartSh.subCartCrys}
Let $X$ be an $\Fr$-finite scheme and $M$ a coherent Cartier module on $X$ with associated crystal $\CM$. Then the assignment $N\mapsto\CN$ from the submodules $N$ of $M$ which satisfy $C(N)=N$ to the subcrystals of $\CM$, sending $N$ to its associated crystal, is a bijection. Moreover this assignment is inclusion preserving, \ie, if $N$ and $N'$ are submodules of $M$ with $C(N)=N$ and $C(N')=N'$ representing the subcrystals $\CN$ and $\CN'$ of $\CM$, respectively, then $N\subset N'$ if and only if $\CN$ is a subcrystal of~$\CN'$.
\end{prop}
\begin{proof}
If $N\subset M$ is any Cartier submodule representing the subcrystal $\CN$ of $\CM$, then $\underline{N}\subset N$ will also represent $\CN$. Hence the assignment in the proposition is surjective. Suppose now that two Cartier submodules $N$ and $N'$ of $M$ with $C(N)=N$ and $C(N')=N'$ both represent $\CN$. Then $N+N'$ will also be a Cartier submodule of $M$ with $C(N+N')=N+N'$ which represents $\CN$. Thus we may assume that $N\subset N'$.
But then the injective homomorphisms $N\into N'$ is a nil-isomorphisms, and in particular $N'$ will have the nilpotent quotient $N'/N$. Since $C(N')=N'$ we deduce $N'/N=0$. Hence the assignment in the proposition is also injective. The argument employed in the proof of the injectivity, also easily proves that the assignment is inclusion preserving.
\end{proof}
 The following result is now immediate:
 \begin{cor}
\label{t.FiniteSubmodules}
Let $X$ be an $\Fr$-finite scheme. A coherent Cartier module $M$ on $X$ has only finitely many submodules $N$ such that $\Ca(N)=N$.
\end{cor}
\begin{rem}
\label{r.FiniteSubmodules}
If we call two submodules $N,N' \subseteq M$ \emph{equal up to nilpotence} if $\underline{N}=\underline{N'}$ as submodules of $M$, then \autoref{t.FiniteSubmodules} says  that a coherent Cartier module has \emph{up to nilpotence}, only finitely many Cartier submodules.
\end{rem}

\subsection{Nil-decomposition series}
Our main result \autoref{t.FiniteLenth} and its \autoref{t.CrysFiniteLength} state that in the category of Cartier crystals every object has finite length. Via standard arguments one derives from this a Jordan-Holder theory for Cartier crystals. For Cartier modules, this yields in turn a theory of Jordan-Holder series \emph{up to nilpotence}. Let us spell out what this means. Let $M$ be a coherent Cartier module. Then a \emph{nil-decomposition series} is a sequence
\[
    M = M_0 \supseteq M_1 \supseteq M_2 \supseteq \ldots M_{n-1} \supseteq M_n =0
\]
of submodules of $M$ such that for each $i$ the quotients $M_i/M_{i+1}$ are not nilpotent (equivalently $(M_i/M_{i+1})_{min} \neq 0$). A nil-decomposition series is called maximal, if and only if it cannot be refined to a longer nil-decomposition series. Clearly, a nil-decomposition series is maximal if and only if all quotients $(M_i/M_{i+1})_{min}$ are simple. We define the \emph{quasi-length} of $M$, denoted $\ql(M)$, as the length of the shortest maximal nil-decomposition series of $M$. Clearly $\ql(M)$ is equal to the length of the associated Cartier crystal $\CM$. As an immediate consequence of \autoref{t.subCartSh.subCartCrys}, we deduce:
\begin{prop}\label{t.BijJordanHolder}
Let $M$ be a coherent Cartier module on the $\Fr$-finite scheme $X$ with $M=C(M)$. Then there is a bijection between nil-decomposition series of $M$ satisfying $C(M_i)=M_i$ for all $i$ and decomposition series of the crystal represented by~$M$.
\end{prop}
This yields the following type of nil-decomposition series for $M$:
\begin{prop}
\label{t.ExlicitJordanHolder}
Let $M$ be a coherent Cartier module on the $\Fr$-finite scheme $X$.
Then there exists a finite sequence
\[
    M=M_0 \supseteq \underline{M_0} \supseteq M_1 \supseteq \underline{M_1} \supseteq \ldots M_{n-1} \supseteq \underline{M_{n-1}} \supseteq M_n \supseteq \underline{M_n} = 0
\]
such that $M_i/\underline{M_i}$ is nilpotent (and possibly zero) and $\underline{M_i}/M_{i+1}$ is not nilpotent and simple (equiv. minimal, simple, and nonzero)
\end{prop}
\begin{proof}
Let $ \underline{M} = \underline{M_0} \supseteq \underline{M_1} \supseteq \ldots \supseteq \underline{M_{n-1}} \supseteq \underline{M_n} = 0$ be a decomposition series of $\underline{M}$ corresponding by the previous proposition to a Jordan-H\"older series of the crystals represented by $M$. Defining $M_i$ as the pullback under $\underline{M_{i+1}}\to \underline{M_{i+1}}/\underline{M_i}$ of $(\underline{M_{i+1}}/\underline{M_i})_{\nil}$, the result follows.
\end{proof}
The following theorem follows now formally from the finite length \emph{up to nilpotence} via standard arguments in Jordan-Hölder theory.
\begin{thm}
Let $M$ be a coherent Cartier module on the $\Fr$-finite scheme $X$.
\begin{enumerate}
\item Every nil-decomposition series of $M$ can be refined to a maximal nil-decomposition series of~$M$.
\item Every maximal nil-decomposition series of $M$ has the same length.
\item The minimal simple sub-quotients $(M_i/M_{i+1})_{min}$ are unique (up to re-ordering)
\end{enumerate}
\end{thm}

\section{Co-finite left $R[F]$-modules and Lyubeznik's $\mathcal{F}$-finite modules}
\label{s.applications}

The category of Cartier modules can be viewed as central axis which relates other categories of $\CO_X$-modules with other types of Frobenius related actions on them. The picture that one has is roughly as follows:
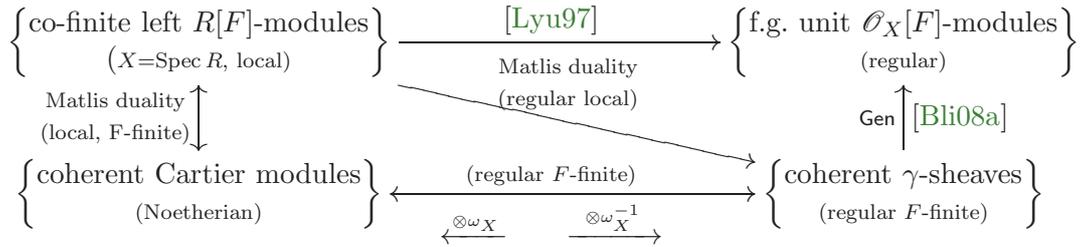
\[
\xymatrix@C=10pc{
    {\begin{Bmatrix}\text{co-finite left $R[F]$-modules}\\ {(\scriptstyle X=\Spec R, \text{ local})}\end{Bmatrix}}
    \ar[r]^{\txt{\cite{Lyub}}}
    \ar@{<->}[d]_{\txt{\scriptsize Matlis duality \\ \scriptsize (local, F-finite)}}
    \ar[rd]^{\txt{\scriptsize Matlis duality \\ \scriptsize(regular local)}} &
    {\begin{Bmatrix}\text{f.g. unit $\CO_X[F]$-modules}  \\  {\scriptstyle (\text{regular})} \end{Bmatrix}} \\
    {\begin{Bmatrix} \text{coherent Cartier modules} \\ {\scriptstyle (\text{Noetherian})} \end{Bmatrix}}
    \ar@{<->}[r]^{\scriptstyle (\text{regular $F$-finite})}_{\xleftarrow{\tensor \omega_X}{} \qquad \to[\tensor \omega_X^{-1}]} &
    {\begin{Bmatrix}\text{coherent $\gamma$-sheaves} \\ {\scriptstyle (\text{regular $F$-finite})} \end{Bmatrix}}
    \ar[u]^{\Gen}_{\txt{\cite{Bli.MinimalGamma}}}
    }
\]
The precise categories and the arrows between them will be explained in this section. The parenthesized assumptions in the diagram denote the generality in which the respective category or functor exists. It turns out that all arrows, except the ones to the top right corner, are equivalences of categories; these remaining two arrows become equivalences when the source is replaced by the respective category of crystals, \ie~after killing nilpotence. The arrows in the bottom left triangle all preserve nil-isomorphisms and hence induce an equivalence on the level of the associated crystals.

In addition to this -- and this is a main point in the manuscript \cite{BliBoe.CartierCrys} -- Grothendieck-Serre duality induces an equivalence between Cartier modules and so called $\tau$-sheaves of \cite{BoPi.CohomCrys}, \ie~left $\CO_X$-modules. However this equivalence is on the level of an appropriate derived category and hence quite more technical than the current paper. In particular, the preservation of certain functors, which we only hinted at here, is a non-trivial matter. We will not discuss this (related) viewpoint here but refer the reader to the upcoming \cite{BliBoe.CartierCrys}, or to \cite[Chapter 6]{Bli.habil} for an overview of these results.

\subsection{Co-finite modules with Frobenius action}
In this section we explain the relationship of our theory with the theory of co-finite left $R[F]$-modules for a local ring $(R,m)$ as studied before by many authors, mainly in connections to questions about local cohomology, see for example \cite{HaSp,Sharp.GradedAnnihil,Smith.rat,LyubSmith.Comm,Hara,Lyub}. The relation to our theory is via Matlis duality, quite similar to the duality in the case of a field we described above.

\subsubsection{Matlis duality}
Let $(R,\frm)$ be complete, local and $F$-finite. Denote by $E=E_R$ an injective hull of the residue field $R/\frm$ of $R$. Since $R$ is $F$-finite one has that $\Fr^\flat E_R = \Hom_R(\Fr_*R,E_R) \cong E_{\Fr_*R}$, and since $R$ and $\Fr_*R$ are isomorphic as rings we may identify $E_{\Fr_*R}$ with $E_R$, and hence get an isomorphism $E_R \cong \Fr^\flat E_R$. As in the case of fields treated in \autoref{s.CartierFields}, we fix from now on an isomorphism $\Fr^\flat E_R \cong E_R$; note that any other choice of  isomorphism $\Fr^\flat E_R \cong E_R$ will differ from our fixed one by an automorphism of $E_R$, \ie~by multiplication with a unit in $R$. We denote by $(\usc)^\vee_R=\Hom_R(\usc,E_R)$ the Matlis duality functor.
\begin{lem}
\label{t.F_*commutesMatils}
Let $(R,\frm)$ be local and $F$-finite. Then there is a functorial isomorphism $\Fr_*(\usc)^\vee_R \cong (\Fr_*\usc)^\vee_R$.
\end{lem}
\begin{proof}
Using the duality for finite morphisms we get for any $R$-module $M$
\[
    \Hom_R(\Fr_*M,E_R) \cong \Fr_* \Hom_R(M,\Fr^\flat E_R) \cong \Fr_* \Hom_R(M,E_R)
\]
where the final isomorphism is induced by the fixed isomorphism above.
\end{proof}
\begin{prop}
\label{t.MatlisDual}
Let $(R,\frm)$ be complete, local and $F$-finite. Then Matlis duality induces an equivalence of categories between left $R[\Fr]$-modules which are co-finite\footnote{by co-finite we mean her Artinian as an $R$-module} as $R$-modules and coherent Cartier modules, \ie~$R$-finitely generated right $R[F]$-modules. The equivalence preserves nilpotence, and hence, since duality is exact, it also preserves nil-isomorphism.
\end{prop}
\begin{proof}
If $M$ is a finitely generated Cartier module, then the dual of the structural map, together with the functorial isomorphism of the preceding lemma gives a map
\[
    M^\vee \to (\Fr_*M)^\vee \cong \Fr_*(M^\vee)
\]
which is nothing but a left action of $R[\Fr]$ on the co-finite $R$-module $M^\vee$. Conversely, the same construction works and it is easy to check that this indeed induces an equivalence of categories.
\end{proof}
This correspondence precisely explains the relationship between our \autoref{t.finitelymanyirred} and the result of Enescu and Hochster \cite[Theorem 3.6]{EnescuHochster.FrobLocCohom} discussed above in \autoref{r.EnHoFiniteness}. Further consequences are:
\begin{thm}
\label{t.FiniteLengthCofinite}
Let $(R,\frm)$ be local and $F$-finite. Let $N$ be a co-finite $R$-module with a left $R[F]$-action. Then
\begin{enumerate}
\item $N_{\nil}=\{n \in N | \Fr^e(n)=0 \text{ for some } e\}$ is nilpotent, \ie~there is $e>0$ such that $\Fr^e(N_{\nil})=0$, see \cite[Proposition 1.11]{HaSp}.
\item Up to nilpotence, $N$ has finite length, \ie~every increasing chain of submodules of $N$ eventually becomes nil-constant (\ie~successive quotients are nilpotent). $N$ has a nil-decomposition series, see \cite[Theorem 4.7]{Lyub}.
\item Up to nilpotence, $N$ has only finitely many $R[F]$-submodules. More precisely, $N$ has only finitely many submodules for which the action of $F$ on the quotient is injective.
\end{enumerate}
\end{thm}
\begin{proof}
    All the statements follow via the just described duality from the corresponding statements for coherent Cartier modules (\autoref{t.ImagesStabilize}, \autoref{t.FiniteLenth}, \autoref{t.FiniteSubmodules} and \autoref{r.FiniteSubmodules}).
\end{proof}

In \cite{EnescuHochster.FrobLocCohom}, Enescu and Hochster derive under certain purity and Gorenstein conditions the finiteness of the actual number (and not just their number up to nilpotence) of $R[F]$-submodules of the top local cohomology module $H^d_{\frm}(R)$. We will give a simple version of this for Cartier modules, and show how this implies a slight generalization of results in \cite{EnescuHochster.FrobLocCohom}, see \cite[Discussion 4.5]{EnescuHochster.FrobLocCohom}. This is directly related to \autoref{r.CompatiblySplit} on compatibly split subvarieties.
\begin{prop}
\label{t.FiniteCartFsplit}
Let $R$ be $F$-finite and $F$-split, \ie~there is a map $\Ca:\Fr^e_*R \to R$ splitting the Frobenius $\Fr: R \to \Fr^e_*R$. Then the Cartier module $(R,\Ca)$ has only finitely many Cartier submodules.
\end{prop}
\begin{proof}
Let $I \subseteq R$ be a Cartier submodule, \ie~an ideal such that $C(I) \subseteq I$. Then since $\Ca$ splits $\Fr$ we have $I=\Ca(\Fr(I))=\Ca(I^{[q^e]}) \subseteq \Ca(I) \subseteq I$. This shows that any Cartier submodule has the property that $\Ca(I)=I$, hence the claim follows by \autoref{t.FiniteSubmodules}.
\end{proof}
In analogy with a definition in \cite{EnescuHochster.FrobLocCohom} we say that a Cartier module $(M,C)$ is \emph{anti-nilpotent} if for all Cartier submodules $N\subseteq M$ we have $C(N)=N$. For such \autoref{t.FiniteSubmodules} yields:
\begin{prop}
\label{t.antinil-and-submod}
An anti-nilpotent Cartier module has only finitely many Cartier submodules.\footnote{The converse does not hold, as the example of a simple $\CO_X$-module with zero structural map shows.}
\end{prop}
It is not difficult to check that anti-nilpotent is equivalent to each of the following conditions.\footnote{Anti-nilpotent means that all Cartier submodules have surjective structural map. The surjectivity of the structural map, however, always passes to Cartier module quotients, hence if $M$ is anti-nilpotent, then all sub-quotients have surjective structural map. In particular there are no nilpotent sub-quotients, \cf~\cite[Proposition 4.6]{EnescuHochster.FrobLocCohom}}
(a) $M$ does not have non-trivial nilpotent Cartier sub-quotients.
(b) $M$ does not have non-trivial Cartier sub-quotients with zero structural map.
(c) $M$ and all its Cartier submodules and quotients are minimal.
(d) $M$ and all its Cartier submodules and quotients are anti-nilpotent.

Specializing to the case that $X=\Spec R$ with $R$ a complete local ring, we obtain from \autoref{t.FiniteCartFsplit} via Matlis duality in \autoref{t.MatlisDual} that a coherent Cartier module $M$ is anti-nilpotent if its Matlis dual $M^\vee$ is an anti-nilpotent left $R[\Fr]$-module, meaning that $M^\vee$ has no nilpotent $R[F]$-sub-quotients, \cf~\cite[Definition 4.7]{EnescuHochster.FrobLocCohom}. We obtain the following extension of \cite[Corollary 4.17]{EnescuHochster.FrobLocCohom}.
\begin{prop}
\label{t.AntiNilpSplit}
Let $(R,\frm)$ be complete, local, $F$-finite, and $F$-split.
\begin{enumerate}
\item \label{z.AntiNilpSplit.b}The injective hull of the residue field $E_R$ has \emph{some} left $R[F]$-structure for which $E_R$ is anti-nilpotent.
\item \label{z.AntiNilpSplit.c}If $R$ is also quasi-Gorenstein (i.e. $H^d_\frm(R) \cong E_R$), then the top local cohomology module $H^d_\frm(R)$ with its \emph{canonical} left $R[F]$-module structure is anti-nilpotent and hence by \autoref{t.antinil-and-submod} has only finitely many $R[F]$-submodules, \cf~\cite[Theorem 3.7, Corollary 4.17]{EnescuHochster.FrobLocCohom}.
\end{enumerate}
\end{prop}
\begin{proof}
For part \autoref{z.AntiNilpSplit.b}, let $C: \Fr_*R \to R$ be a splitting of the Frobenius. In the proof of \autoref{t.FiniteCartFsplit} it was observed that the Cartier module $(R,C)$ is anti-nilpotent. The duality of \autoref{t.MatlisDual} induces an $R[\Fr]$-module structure $C^\vee: E_R \to \Fr_* E_R$ and the anti-nilpotence of $(R,C)$ immediately translates into the anti-nilpotence of $(E_R,C^\vee)$.

Let us show \autoref{z.AntiNilpSplit.c}. Abbreviating $H=H^d_\frm(R)$, we have $H=H^d_\frm(R) \cong E_R$ since $R$ is quasi-Gorenstein. Part \autoref{z.AntiNilpSplit.b} shows that $(H,C^\vee)$ is an anti-nilpotent left $R[\Fr]$-module. At the same time, by the functoriality of local cohomology the Frobenius on $R$ induces another (natural) Frobenius action $F_H:H \to \Fr_*H$. We claim that $C^\vee$ factors through $F_H$, \ie, that there is an $r \in R$ such that $C^\vee=r\cdot F_H$. Suppose that this claim holds. Then every $R[F]$-submodule $N \subseteq H$ under the action coming from $F_H$ is also an $R[F]$-submodule of $H$ under the action coming from $C^\vee$, because $C^\vee(N) = rF_H(N) \subseteq rN \subseteq N$. Hence the subquotients of $H$ for the $F_R$-action are a subset of those for the $C^\vee$-action. By \autoref{z.AntiNilpSplit.b} none of the non-zero subquotients in the latter set is nilpotent for the action of $C^\vee$. The factorization $C^\vee=rF_H$ implies the same for the subquotients in the former set for the action of~$F_H$, proving the anti-nilpotence of $(H,F_H)$.

It remains to show the factorization $C^\vee=rF_H$. Since the top local cohomology is the cokernel of an appropriate \Cech~complex one easily observes that the adjoint (under adjointness for $\Fr^*$ and $\Fr_*$) of $F_H$ is an isomorphism, see \cite[Example 2.7]{Bli.PhD}. If $\theta':\Fr^*H \to H$ denotes the adjoint of $C^\vee$, the composition $\theta'\circ\theta^{-1}$ is an $R$-linear endomorphism of $H$. Since $R$ is complete, and $H \cong E_R$ we have that $\theta'\circ\theta^{-1}$ is given by multiplication by an element $r \in R$. Hence $\theta'=r\theta$ which implies for their adjoints the claimed equality $C^\vee = r F_H$.
\end{proof}

\subsection{Lyubeznik's $\Fr$-finite modules}
The connection of our theory of Cartier modules with Lyubeznik's theory of $\mathcal{F}$-finite modules is via his notion of roots, or generators. This has been worked out in quite some detail in \cite{Bli.MinimalGamma} where the first author introduces a category of $\gamma$-sheaves (corresponding to Lyubeznik's generators) and shows that the category of $\gamma$-crystals ($\gamma$-sheaves modulo nilpotence) is equivalent to Lyubeznik's $\mathcal{F}$-finite modules. In this section we show that under a reasonable hypothesis ($X$ is regular, $F$-finite and sufficiently affine) the category of Cartier modules is equivalent to the category of $\gamma$-sheaves and the equivalence preserves nilpotence. This is a slight variant of work in \cite{BliBoe.CartierCrys} where the case $X$ regular and essentially of finite type over an $F$-finite field is treated.

\subsubsection{Cartier modules and \texorpdfstring{$\gamma$}{gamma}-sheaves}
In \cite{Bli.MinimalGamma} the first author introduced -- motivated by Lyubeznik's concept of a root in \cite{Lyub} -- the category of $\gamma$-sheaves. For a regular and $F$-finite scheme $X$, this is the category consisting of quasi-coherent $\CO_X$-modules $N$ equipped with an $\CO_X$-linear map $\gamma: N \to \Fr^*N$. The theory develops quite analogously as in the case of Cartier modules, taking the viewpoint that a Cartier module is given by a map $\kappa: M \to \Fr^\flat M$ and replacing $\Fr^\flat$ by $\Fr^*$. In particular, there is a notion of nilpotence, meaning that $\gamma^i=0$ for some $i$, where $\gamma^i=\Fr^*\gamma^{i-1}\circ \gamma$. Nilpotent $\gamma$-sheaves form a Serre sub-category of the abelian category of $\gamma$-shaves, there is an abelian category of $\gamma$-crystals and so forth. A reason why one has to assume regularity for $X$ in the context of $\gamma$-sheaves is because one needs the exactness of $\Fr^*$ to guarantee that the kernel of a map of $\gamma$-sheaves naturally is also a $\gamma$-sheaf. In the case of Cartier modules, the exactness of $\Fr_*$ holds in general since $\Fr$ is an affine morphism. The passage from $\gamma$-sheaves to Cartier modules is achieved by tensoring with the dualizing sheaf $\omega_X$. We recall the necessary facts below.
\begin{lem}
\label{t.fflatForflat}
Let $f: Y \to X$ be a finite flat morphism and $M$ a quasi-coherent $\CO_X$-module, then there is a functorial isomorphism
\[
    f^\flat \CO_X \tensor_{\CO_Y} f^* M \to[\cong] f^\flat M
\]
\end{lem}
\begin{proof}
One has an $\CO_Y$-linear map given by sending a section $\phi \tensor_{\CO_Y} s \tensor_{\CO_X} m$ to the map sending a section $r \in \CO_Y$ to $\phi(sr)\cdot m$. To verify that it is an isomorphism we may assume that $Y=\Spec S$, $X=\Spec R$ are affine and $S$ is a finitely generated and free $R$-module. Then everything comes down to checking that the homomorphism
\[
    \Hom_R(S,R) \tensor_R M \to[\phi \tensor n \mapsto (r \mapsto \phi(r)n)] \Hom_R(S,M).
\]
is bijective, which is easily verified since $S$ is finite and free over $R$.
\end{proof}
Applying this lemma to $M=\omega_X$, a dualizing sheaf, we obtain as an immediate corollary:
\begin{cor}
Let $f: Y \to X$ be finite and flat, and suppose that $\omega_X$ is invertible. Then
\begin{align*}
f^\flat \CO_X & \cong f^\flat \omega_X \tensor f^*\omega_X^{-1} \\
f^\flat(\omega_X \tensor M) &\cong f^\flat\omega_X \tensor f^*M \\
f^*(\omega_X^{-1} \tensor M) &\cong  (f^\flat\omega_X)^{-1}  \tensor f^\flat M .
\end{align*}
If we assume in addition that $f^\flat \omega_X \cong \omega_Y$, then
\begin{align*}
\phantom{the relative dualizing}f^\flat\CO_X &=\omega_{Y/X}\text{ the relative dualizing sheaf}\\
f^\flat(\omega_X \tensor M) &\cong \omega_Y \tensor f^*M \\
f^*(\omega_X^{-1} \tensor M) &\cong  \omega_Y^{-1}  \tensor f^\flat M .
\end{align*}
\end{cor}
Note that the additional assumption that $f^\flat \omega_X \cong \omega_Y$ is satisfied if $X$ is normal and either $X$ is essentially of finite over a local Gorenstein ring, or $X$ is sufficiently affine, as explained in \autoref{s.dualizingsheaf}.
\begin{thm}
\label{t.CartierGammaEquiv}
Let $X$ be regular and $F$-finite, and assume that there is a dualizing sheaf $\omega_X$ such that $F^\flat\omega_X \cong \omega_X$. Then the category of Cartier modules on $X$ is equivalent to the category of $\gamma$-sheaves on $X$.

The equivalence is given by tensoring with $\omega^{-1}_X$, its inverse by tensoring with $\omega_X$. This preserves coherence, nilpotence, and nil-isomorphism, and hence induces an equivalence between Cartier crystals and $\gamma$-crystals.
\end{thm}
\begin{proof}
Since, if $X$ is regular, the Frobenius $\Fr$ is flat by \cite{Kunz}, we may apply the preceding corollary to obtain isomorphisms
\[
    \omega_X^{-1} \tensor \Fr^\flat M \cong \Fr^*(\omega_X^{-1} \tensor M) \text{ and }
    \omega_X \tensor \Fr^* N \cong \Fr^\flat(\omega_X \tensor N)
\]
This shows that if $M \to[\kappa] \Fr^\flat M$ is a Cartier module, then
\[
    \omega_X^{-1} \tensor M \to[\omega_X^{-1}\tensor \kappa] \omega_X^{-1} \tensor \Fr^\flat M \cong \Fr^*(\omega_X^{-1} \tensor M)
\]
gives $\omega_X^{-1} \tensor M$ a natural structure of a $\gamma$-sheaf. Conversely, if $N \to[\gamma] \Fr^* N$ is a $\gamma$-sheaf, then
\[
    \omega_X \tensor N \to[\omega_X \tensor \gamma] \omega_X \tensor \Fr^* N \cong \Fr^\flat(\omega_X \tensor N)
\]
equips $\omega_X \tensor N$ with the structure of a Cartier module. It is immediate that these two operations are inverse to one another. By functoriality, nilpotence is clearly preserved (nilpotence is $\kappa^i=0$, resp. $\gamma^i=0$, which is preserved by a functor) and since tensoring with a locally free module is exact, nil-isomorphisms are also preserved. Hence one gets an induced equivalence on the level of crystals.
\end{proof}
As a corollary of this equivalence our \autoref{t.ExistenceMinimal} yields the main result of \cite{Bli.MinimalGamma} on the existence of minimal $\gamma$-sheaves.
\begin{thm}[\cite{Bli.MinimalGamma} Theorem 2.24]
Let $X$ be regular and $F$-finite. For each coherent $\gamma$-sheaf $M$  there is a unique (functorial) minimal $\gamma$-sheaf $M_{min}$ which is nil-isomorphic to $M$.
\end{thm}
\label{t.minimalGamma}
\begin{proof}
Minimality for $\gamma$-sheaves is defined in the same way as for Cartier modules, namely, a $\gamma$-sheaf is called \emph{minimal} if it has neither nilpotent submodules or quotients. Since the equivalence in \autoref{t.CartierGammaEquiv} is exact and preserves nilpotence it follows immediately that minimality is also preserved. It was observed in \cite[Lemma 2.17]{Bli.MinimalGamma} that minimality for $\gamma$-sheaves localizes, and that minimal $\gamma$-sheaves are unique in their nil-isomorphism class \cite[Proposition 2.25]{Bli.MinimalGamma}. It follows that one can reduce the proof of the existence to the members of a finite affine cover of $X$. Hence we may assume that $X$ is affine and furthermore that $\Fr^\flat \omega_X \cong \omega_X$. In this situation we may apply  \autoref{t.CartierGammaEquiv} to \autoref{t.ExistenceMinimal} to derive the result.
\end{proof}
As a translation of our finite length result \autoref{t.FiniteLenth} for coherent Cartier crystals and of \autoref{t.finiteEndCrys} we get the following statement for coherent $\gamma$-sheaves.
\begin{thm}
\label{t.FiniteLengthGamma}
Let $X$ be regular and $F$-finite. Then every coherent $\gamma$-sheaf has, up to nilpotence, finite length in the category of $\gamma$-sheaves. In other words, the category of $\gamma$-crystals is Artinian (DCC) and Noetherian (ACC). 

Moreover, the endomorphism ring of any simple non-zero $\gamma$-crystal is a finite field containing $\Primefield_q$.
\end{thm}
\begin{proof}
The Noetherian-ness is clear, since already as $\CO_X$-modules a coherent $\gamma$-sheaf satisfies the ascending chain condition on submodules (we always assume that $X$ is locally Noetherian). Hence we have to show that any descending chain of $\gamma$-sheaves stabilizes (up to nilpotence). But this may be checked on a finite affine cover. Hence we may assume that $X$ is affine and furthermore one has an isomorphism $\Fr^\flat \omega_X \cong \omega_X$. This enables us to employ \autoref{t.CartierGammaEquiv} which reduces the statement to the finite length result for coherent Cartier modules (up to nilpotence) shown in \autoref{t.FiniteLenth}.

For the last assertion observe first that, by the lemma of Schur, the endomorphism ring of a simple non-nilpotent $\gamma$-crystal is a skew field. Next note that any endomorphism on $X$ induces endomorphisms on the restrictions to any open of a fixed finite affine cover. On sufficiently small affine open sets we can apply  \autoref{t.CartierGammaEquiv}, and so there the endomorphism rings are finite. But the map sending a global endomorphisms to its restrictions on an affine cover is clearly injective.
\end{proof}

\subsubsection{Finitely generated unit $R[\Fr]$-modules.}
There is a functor from (coherent) $\gamma$-sheaves to the category of (locally finitely generated) unit $\CO_X[\Fr]$-modules. This latter category was introduced in \cite{Lyub}, under the name of $\mathcal{F}$-finite modules, in the regular affine case, and in \cite{EmKis.Fcrys} for regular schemes of finite type over a field. It consist of $\CO_X$-quasi-coherent left $\CO_X[\Fr]$-modules $\CM$ which are locally finitely generated over $\CO_X[\Fr]$, and such that the adjoint map to the $\Fr$-action
\[
    \theta: \Fr^*\CM \to \CM
\]
is an isomorphism. Already in \cite{Lyub} it was observed that there is a functor from coherent $\gamma$-sheaves to finitely generated unit $\CO_X[F]$-modules. This functor, denoted $\Gen$, sends a $\gamma$-sheaf $M \to[\gamma] \Fr^*M$ to the limit $\Gen(M)$ of the directed system
\[
    M \to[\gamma] \Fr^*M \to[\Fr^*\gamma] \Fr^{2*}M \to \ldots
\]
It is shown in \cite{Bli.MinimalGamma} that this functor induces an equivalence of categories from $\gamma$-crystals to finitely generated unit $\CO_X[\Fr]$-modules.
\begin{prop}[\cite{Bli.MinimalGamma} Theorem 2.7]\label{t.equivGammaUnit}
Let $X$ be regular and $F$-finite. Then the Functor $\Gen$ from coherent $\gamma$-sheaves on $X$ to finitely generated unit $\CO_X[F]$-modules induces an equivalence
\[
    \{\text{$\gamma$-crystals}\} \to[\simeq] \{\text{f.g. unit $\CO_X[F]$-modules}\}.
\]
\end{prop}
As an immediate application of \autoref{t.equivGammaUnit} and \autoref{t.FiniteLengthGamma} we obtain the following generalization of the main result of \cite[Theorem 3.2]{Lyub}.
\begin{thm}
\label{t.FiniteLenthFGunit}
Let $X$ be regular and $F$-finite, then every finitely generated unit $\CO_X[\Fr]$-module has finite length in the category of unit $\CO_X[\Fr]$-modules. Moreover the endomorphism ring of any simple finitely generated unit $\CO_X[\Fr]$-module is a finite field containing~$\Primefield_q$.
\end{thm}

\begin{rem}
Lyubeznik shows in \cite[Theorem 3.2]{Lyub} that the finite length for finitely generated unit $\CO_X[F]$-modules holds in the case that $X$ is regular and essentially of finite type over a regular local ring. So the result just given extends this to all $F$-finite schemes, but does not completely recover Lyubeznik's result. Important cases which are covered by \cite[Theorem 3.2]{Lyub} but not our result is that of a finite type scheme $X$ over a field with $[k:k^p] =\infty$, or $X = \Spec R$ with $R$ local but not $F$-finite. We suspect that the main result in this paper, the finite length of Cartier crystals \autoref{t.FiniteLenth}, will also hold in these cases. However, our proof, as well as the transition from Cartier modules to finitely generated unit $\CO_X[F]$-modules, is closely tied to the $F$-finite-ness there seems to be some different techniques necessary to obtain these results.
\end{rem}

Note also that in combining \autoref{t.CartierGammaEquiv} with the above quoted \cite[Theorem 2.7]{Bli.Cart} we obtain
\begin{thm}
Let $X$ be regular and $F$-finite, and assume that there is a dualizing sheaf $\omega_X$ such that $F^\flat\omega_X \cong \omega_X$. Then the category of Cartier modules on $X$ is equivalent to the category of finitely generated unit $\CO_X[\Fr]$-modules.
\end{thm}

\bibliographystyle{amsalpha}
\bibliography{./../CommonFiles/MyBibliography}

\end{document}